\documentclass[10pt]{amsart}
\usepackage{amsmath,amsfonts,amssymb,amsthm}
\usepackage{graphics,graphicx}
\usepackage[colorlinks,hyperindex,linkcolor=blue,
citecolor=purple]{hyperref}
\hypersetup{
pdfauthor = {Francis Bischoff, Alvaro del Pino Gomez, Aldo Witte},
 pdftitle= {Jets of foliations and  algebroids},
 pdfsubject = {Lie Theory, Subriemannian Geometry, tangent distributions, log-Geometry, Lie algebroids}}
\usepackage{tikz-cd}
\usepackage{tikz}
\usepackage[all]{xy}
\usepackage{subfiles}
\usetikzlibrary{matrix,arrows,
decorations.pathmorphing}

\newcommand{\ff}{\mathcal{F}}

\newcommand{\A}{{\mathcal{A}}}

\newcommand{\Gcal}{{\mathcal{G}}}

\newcommand{\Ical}{{\mathcal{I}}}

\newcommand{\Ncal}{{\mathcal{N}}}
\newcommand{\Ocal}{{\mathcal{O}}}

\newcommand{\RR}{\mathbb{R}}

\newcommand{\Gr}{\operatorname{Gr}}

\newcommand{\Hom}{\operatorname{Hom}}

\newcommand{\id}{\textbf{\textit{I}}}

\newcommand{\End}{\operatorname{End}}

\newcommand{\GL}{\operatorname{GL}}

\newcommand{\Diff}{\operatorname{Diff}}

\newcommand{\rank}{\operatorname{rank}}

\newcommand{\hol}{{\operatorname{hol}}}

\newcommand{\Fol}{{\operatorname{Fol}}}

\newcommand{\Tangency}{{\operatorname{Tan}}}
\newcommand{\Tan}{{\operatorname{Tan}}}

\newcommand{\RH}{{\operatorname{RH}}}
\newcommand{\Mor}{{\operatorname{Hom}}}
\newcommand{\Char}{{\operatorname{Char}}}
\newcommand{\Fr}{{\operatorname{Fr}}}

\newcommand{\set}[1]{\lbrace #1\rbrace}
\newcommand{\inp}[1]{\langle #1\rangle}

\newcommand{\rr}{\mathbb{R}}

\DeclareMathOperator{\Res}{Res}

\newcommand{\N}{\mathcal{N}}

\newcommand{\C}{\mathcal{C}}

\DeclareMathOperator{\Der}{Der}
\DeclareMathOperator{\At}{At}
\DeclareMathOperator{\Aut}{Aut}
\DeclareMathOperator{\Hol}{Hol}
\DeclareMathOperator{\Var}{Var}
\DeclareMathOperator{\Sym}{Sym}

\renewcommand{\hat}{\widehat}

\DeclareMathOperator{\vertrm}{vert}

\newcommand{\tsigma}{\tilde{\sigma}}
\DeclareMathOperator{\ad}{ad}

\usepackage{thmtools}
\declaretheorem[style=definition,qed=$\diamondsuit$]{definition}
\declaretheorem[style=definition,qed=$\triangle$,sibling=definition]{example}
\declaretheorem[style=plain,sibling=definition]{theorem}
\declaretheorem[style=plain,sibling=definition]{lemma}
\declaretheorem[style=plain,sibling=definition]{proposition}
\declaretheorem[style=plain,sibling=definition]{corollary}
\declaretheorem[style=definition,qed=$\diamondsuit$,sibling=example]{claim}
\declaretheorem[style=definition,sibling=example]{question}
\declaretheorem[style=definition,qed=$\diamondsuit$,sibling=claim]{remark}
\declaretheorem[style=definition,sibling=example]{problem}

\newtheorem*{claim*}{Claim}
\numberwithin{theoremalpha}{section}

\numberwithin{equation}{section}
\numberwithin{definition}{section}
\numberwithin{theorem}{section}
\numberwithin{proposition}{section}
\numberwithin{lemma}{section}
\numberwithin{example}{section}
\numberwithin{remark}{section}
\numberwithin{corollary}{section}
\numberwithin{question}{section}
\numberwithin{problem}{section}

\newtheoremstyle{named}{}{}{\itshape}{}{\bfseries}{.}{.5em}{\thmnote{#3}#1}
\theoremstyle{named}
\newtheorem*{namedtheorem}{}

\numberwithin{equation}{section}
\address{3737 Wascana Pkwy, Regina, SK S4S 0A2, Canada}
\email{Francis.Bischoff@uregina.ca}

\address{Utrecht University, Department of Mathematics, Budapestlaan 6, 3584 CD Utrecht, The Netherlands}
\email{a.delpinogomez@uu.nl}

\address{Departement of Mathematics, Middelheimlaan 1, 2020 Antwerp, Belgium}
\email{aldowitte@hotmail.nl}

\title{Jets of foliations and $b^k$-algebroids}
\author{Francis Bischoff, \'Alvaro del Pino and Aldo Witte}
\begin{document}
\maketitle

\begin{abstract}
In this article, we introduce and study singular foliations of $b^k$ type. These singular foliations formalize the properties of vector fields that are tangent to order $k$ along a submanifold $W \subset M$. Our first result is a classification of these foliations, relating them to geometric structures defined in a formal neighborhood of the submanifold, such as jets of distributions that are involutive up to order $k-1$. 

When $W$ is a hypersurface, singular foliations of $b^k$ type are Lie algebroids. In this particular case, they are generalizations of the $b^k$-tangent bundles introduced by Scott. Indeed, they are always locally isomorphic to $b^k$-tangent bundles, but globally such an isomorphism is obstructed by a holonomy invariant. Our second main result is a Riemann-Hilbert-style classification of singular foliations of $b^k$ type in terms of holonomy representations. 

In this paper, we study singular foliations of $b^k$ type from several different perspectives. In particular: (1) We study the problem of extending a $k$th-order foliation to a $(k+1)$-th order foliation and prove that this is obstructed by a characteristic class. (2) When $W$ is a hypersurface, we give a detailed study of algebroid differential forms and extend Scott's calculation of the cohomology. (3) We study algebroid symplectic forms in terms of the geometric structures induced on $W$. In particular, we find that there is a close relationship between the above obstruction class for extensions and the symplectic variation of the symplectic foliation induced on $W$. 

\end{abstract}

\tableofcontents

\section{Introduction}

Let $D$ be a smooth hypersurface in a manifold $M$. To $(M,D)$ we can associate the $b$-tangent bundle, the Lie algebroid over $M$ whose sheaf of sections is the sheaf of vector fields on $M$ which are tangent to $D$. This was first introduced in algebraic geometry where it is known as the logarithmic tangent bundle. It's origins go at least as far back as Deligne's work flat connections with regular singularities \cite{MR0417174} and Hodge theory \cite{deligne1971hodge}  and it was further generalized in the work of Saito \cite{saito1980theory} on free divisors. In the $C^{\infty}$-setting it was introduced by Melrose \cite{MR1348401} in the course of his study of index theory on non-compact manifolds. In this setting, the hypersurface $D$ is taken to be the boundary of $M$, thought of as an ``ideal boundary at infinity.'' In both cases, the $b$-tangent bundle is used in order to treat geometric structures on $M$ that are singular along $D$. 

The $b$-tangent bundle is a central example in the theory of Lie algebroids, due both to its simplicity and its usefulness in studying singular geometric structures in a wide variety of settings; see for example \cite{GUILLEMIN2014864, Gualtieri-Li-2012, gualtieri2018stokes, MR3245143, cavalcanti2020self, bischoff2022normal, del2022regularisation}. This has motivated the study of Lie algebroids admitting a larger class of singularities, some examples being \cite{MR3805052, MR4229238, klaasse2018poisson}.

Lie algebroids can be seen as particularly nice examples of singular foliations. Singular foliations appear in many different situations, such as the orbits of Lie groups, the symplectic foliation underlying a Poisson manifold, and within geometric mechanics. The study of singular foliations has received renewed attention within differential geometry following the seminal work of Androulidakis and Skandalis \cite{AS09}.

A crucial class of Lie algebroids for us are the $b^{k+1}$-tangent bundles introduced by Scott in \cite{MR3523250} and studied in depth by Miranda and collaborators \cite{MR4523256, MR4257086, MR4236806, MR3952555}. Informally, its sections are the vector fields that are tangent to $D$ to order $k$.\footnote{In the literature, the $b^{k+1}$-tangent bundle is defined as vector fields $(k+1)$-tangent to $D$. Unfortunately, as we will see in Remark \ref{rem:clash} below, this is inconsistent with the terminology coming from jets. Our indexing for $b^{k+1}$-tangent bundles is consistent with Scott's, but we have a shift by 1 in what we call $k$-tangent.}
However, there is a subtlety here: the notion of $k$-th order tangency is not well-defined. To formalise it, Scott chooses extra data: the $k$-jet of a defining function for $D$. As a result, there is not just one $b^{k+1}$-tangent bundle for a given hypersurface $D \subset M$, but several.

\subsection{Hypersurface Lie algebroids}

In order to understand $b^{k+1}$-tangent bundles better, we abstract their properties:
\begin{definition}
Fix a manifold $M^n$ of dimension $n$ and a hypersurface $D \subset M$. A Lie algebroid $\rho: \A^n \rightarrow M^n$ of rank $n$ is of \textbf{hypersurface} type if its orbits are the components of $D$ and the components of the complement $M\backslash D$.
\end{definition}
This definition can equivalently be stated as the condition that $\rho$ is an isomorphism on the complement of $D$ and has rank $n-1$ over $D$. Given $\rho$, we can consider its determinant $\det(\rho)$ as a section of $\Gamma(\det(A^*) \otimes \det(TM))$. The hypersurface assumption implies that $\det(\rho)$ vanishes precisely along $D$. 
\begin{definition}
A hypersurface Lie algebroid $A \rightarrow M$ is of \textbf{$b^{k+1}$-type} if $\det(\rho)$ vanishes to order $k+1 < \infty$ along $D$.
\end{definition}
Do note that, more generally, we could consider algebroids that have a different order of vanishing in each component of $D$. However, we will henceforth focus on the case in which $D$ is connected, since all interesting behaviours are already present in this setting.

It is immediate that $b^{k+1}$-tangent bundles are Lie algebroids of $b^{k+1}$-type. The question that kick-started the present paper is the converse: \emph{Is every Lie algebroid of $b^{k+1}$-type a $b^{k+1}$-tangent bundle?}

\emph{Locally}, the answer is \emph{yes}:
\begin{proposition}\label{prop:localformhypersurfacealgebroid}
Let $A \rightarrow M$ be a Lie algebroid of $b^{k+1}$-type. Then for any $x \in D$ there are coordinates $(x_1,\ldots,x_n)$ around $D$ such that
\begin{equation}\label{eq:localform}
\Gamma(\A) \simeq \inp{x_1^{k+1}\partial_{x_1},\partial_{x_2},\ldots,\partial_{x_n}}.
\end{equation}
\end{proposition}


But \emph{globally}, the answer turns out to be \emph{no}. The goal of this article is to formalise these claims. We dedicate the rest of the introduction to stating and discussing our main results.

\subsection{$k$-th order foliations}
To understand Lie algebroids of $b^{k+1}$-type we have to introduce the notion of \emph{$k$-th order foliation.} These can be thought of as germs of distributions along $D$, which are \emph{involutive up to order $k$}. We will show that:
\begin{theorem}
There is a one-to-one correspondence between Lie algebroids of $b^{k+1}$-type and $k$-th order foliations.
\end{theorem} 

The Lie algebroids of Scott fit into this framework by letting the $k$-th order foliation $\sigma$ be induced by the foliation of level sets of a defining function around $D$. However, general $k$-th order foliations might have holonomy around $D$. This holonomy is precisely what obstructs Lie algebroids of $b^{k+1}$-type from being globally isomorphic to $b^{k+1}$-tangent bundles.


$k$-th order foliations are interesting objects in their own right, and there is no reason to restrict ourselves to jets along a hypersurface $D$. Many of the results of this paper apply more generally to jets along arbitrary submanifolds $W \subset M$. This forces us to leave the realm of Lie algebroids and work with the following more general family of singular foliations:
\begin{definition}\label{def:typek}
Fix a manifold $M$ and a submanifold $W$. A singular foliation $\ff \subset \mathfrak{X}^1(M)$ is of \textbf{type $b^{k+1}$} if its leaves are the components of $W$ and the components of the complement and, moreover, around every point of $W$, $\ff$ is isomorphic to the following local model
\begin{equation*}
\inp{x^I\partial_{x_1},\ldots,x^I\partial_{x_l},\partial_{x_{l+1}},\ldots,\partial_{x_n}},
\end{equation*}
with $x^I = x_1^{i_1}\cdot \ldots \cdot x_l^{i_l}$ with $I = (i_1,\ldots,i_l)$ running over all multi-indices of length $k+1$. 
\end{definition}

Then:
\begin{namedtheorem}[Proposition \ref{prop:singfolstojets}]
There is a one-to-one correspondence between singular foliations of $b^{k+1}$-type and $k$-th order foliations.
\end{namedtheorem}
The correspondence goes roughly as follows. Given a $k$-th order foliation $\sigma$ along a submanifold $W$, we can consider a germ of distribution $\xi$ along $W$ representing it. To it, we can associate the module of vector fields which are tangent to it up to order $k$:
\begin{equation*}
\Tan^k(M,\xi) = \mathcal{I}_W^{k+1}\mathfrak{X}(M) + \Gamma(\xi),
\end{equation*}
where $\mathcal{I}_W^{k+1}$ is the ideal of functions vanishing to order $k+1$ along $W$ and $+$ denotes the $C^{\infty}(M)$-span of the union. This module is a singular foliation and locally of the form described above.

We will see that there are several equivalent ways of describing $k$-th order foliations, each with their own merits. The first comes from adopting the algebraic viewpoint and considering a \emph{$k$-th order neighbourhood $\N^k_W$} of $W$. One may think of this as a ringed space whose structure sheaf of functions consists of $k$-jets of polynomials in the normal directions of $W$. One may then make sense of \emph{foliations} on ${\rm Spec}(\N^k_W)$.

The second alternative approach comes from the following observation: the singular foliation has a preferred local model and can thus, locally, be described by a defining function. These local defining functions can be combined into a \emph{subsheaf of invariant functions} $\mathcal{C} \subset \mathcal{N}^k_W$.

The final alternative approach comes from observing that, although a singular foliation of \emph{type $b^{k+1}$} is not naturally a Lie algebroid, there is nevertheless a natural Lie algebroid that one can define once a choice of tubular neighbourhood has been fixed. This is given by the $k$-th order Atiyah algebroid of the normal bundle $\At^{k}(\nu_{W})$, defined to be the algebroid of $k$-jets of projectable vector fields on the total space of $\nu_{W}$ which are tangent to the zero section. It turns out that a singular foliation $\mathcal{F}$ of type $b^{k+1}$ is completely encoded by the intersection $\mathcal{F} \cap \At^{k}(\nu_{W})$, which is in turn equivalent to a Lie algebroid morphism $\sigma : TW \to \At^{k}(\nu_{W})$.

The first main result of this article is to show that all these approaches are equivalent:
\begin{namedtheorem}[Theorem \ref{th:equivalence} and Proposition \ref{prop:folsassplittings}]
Let $M^n$ be a manifold together with a connected embedded submanifold $W^{n-l}$ of codimension $l$. Then there is a one-to-one correspondence between:
\begin{enumerate}
\item Singular foliations of type $b^{k+1}$, $\mathcal{A} \subset \mathfrak{X}(M)$.
\item $k$-th order foliations $\sigma \in \Hol^k_W(\Gr(TM,l))$ .
\item Foliations on the $k$-th order neighbrouhood of $W$, $F \subset \Der(\N_W^k)$.
\item Subsheaves of invariant functions $\C \subset \N^k$.
\item And, after fixing a tubular neighbourhood embedding, flat connections $\sigma : TW \rightarrow \At^k(\nu_W)$.
\end{enumerate}

\end{namedtheorem}

\subsection{Holonomy, Frobenius and Riemann-Hilbert}
We will give a precise notion of \emph{holonomy} for $k$-th order foliations, using the Lie algebroid $\At^k(\nu_W)$ of $k$-jets of vector fields along $W$. We view a $k$-th order foliation as a Lie algebroid morphism $\sigma: TW \rightarrow \At^k(\nu_W)$. By invoking Lie's second theorem for Lie algebroid morphisms we obtain a groupoid morphism $\Sigma : \Pi_1(W) \rightarrow \mathcal{G}(\Fr^k(\nu_W))$. By restricting $\Sigma$ to $x \in W$, and choosing a framing for $\nu_{W,x}$ we obtain a representation 
\begin{equation*}
\hol_{\sigma,x} : \pi_1(W,x) \rightarrow J^k\Diff(\rr^l,0)
\end{equation*}
which is the holonomy of $\sigma$. Here $J^k\Diff(\rr^l,0)$ denotes the group of $k$-jets at 0 of diffeomorphisms of $\rr^l$, which fix the origin.


Observe that $\At^1(\nu_W) = \At(\nu_W)$, the Atiyah algebroid of the normal bundle of $W$. Hence a $1$-th order foliation is simply a flat linear connection on $\nu_W$. In this case, the holonomy as defined above is precisely the holonomy of the connection. 

The classical Riemann-Hilbert correspondence states that every linear representation of $\pi_1(W,x)$ corresponds to a flat vector bundle. We will show a higher order version of this result:

\begin{namedtheorem}[Theorem \ref{thm:RH}]
The association $\sigma \mapsto \hol_{\sigma,x}$ induces an equivalence between:
\begin{itemize}
\item  $k$-th order foliations, up to isomorphism
\item conjugacy classes of representions, i.e. points of ${\rm Hom}(\pi_1(W,x),J^k\Diff(\rr^l,0))/J^k\Diff(\rr^l,0)$
\end{itemize}
\end{namedtheorem}
In this theorem, two $k$-th order foliations $\sigma,\sigma'$ are isomorphic if there exists a germ of diffeomorphism $\varphi : \nu_W \rightarrow \nu_W$ which is the identity on $W$ and satisfies $\varphi^*\sigma = \sigma'$. 


The existence of a holonomy representation for $k$-th order foliations also implies the following:

\begin{namedtheorem}[Proposition \ref{prop:simplyconnected}]
Let $W$ be simply connected, and $\sigma$ a $k$-th order foliation. Then, there exists a germ of foliation $\mathcal{F}$ around $W$ with $J^k\mathcal{F} = \sigma$.
\end{namedtheorem}
In particular, this allows us to show that every $k$-th order foliation is locally represented by an actual foliation. Using this we immediately obtain a Frobenius theorem result for $k$-th order foliations (Theorem \ref{th:Frobenius}).

\subsection{The extension problem}
Although every $k$-th order foliation around a symply connected submanifold is representable by an actual foliation, this is not true in general. In Section \ref{sec:extending} we study the following question:
\begin{question}
Let $\sigma$ be a $k$-jet of a foliation. When does there exist a $(k+1)$-th order foliation $\tilde{\sigma}$ with the same $k$-jet as $\sigma$?
\end{question}
We will solve this problem by making use of the theory of Lie algebroid extensions. We will show that:

\begin{namedtheorem}[Proposition \ref{prop:extensionresult}]
A $k$-th order foliation $\sigma$ may be extended to a $(k+1)$-th order foliation if and only if a certain characteristic class $c(\sigma)\in H^2(D;\nu_W^{-k})$ vanishes. 
\end{namedtheorem}
Secondly we will make use of the Riemann-Hilbert correspondence to show:
\begin{namedtheorem}[Theorem \ref{th:secondextensionclass}]
A $k$-th order foliation $\sigma$ may be extended to a $(k+1)$-th order foliation if and only if a certain class in $H^2(\pi_1(D,x);J^k\Diff(\rr^l,0))$ vanishes. 
\end{namedtheorem}
The two approaches will be linked using the Van-Est isomorphism and Morita equivalence of the integrating groupoids.

\subsection{Maurer-Cartan equation}
A $k$-th order foliation always induces a $1$-th order foliation, and consequently a flat connection $\nabla$ on $\nu_W$. One way of thinking about the $k$-th order foliation is as a deformation of this flat connection. This can be made precise as follows:
\begin{namedtheorem}[Proposition \ref{prop:MCexplicit}]
Let $E \rightarrow W$ be a vector bundle with flat linear connection $\nabla$. Then, $k$-th order foliations correspond to tuples $(\eta_0,\ldots,\eta_{k-1})$ with $\eta_i \in \Omega^1(W;\Sym^iE^* \otimes \End(E))$ satisfying the Maurer-Cartan equation
\begin{equation*}
d^{\nabla}\eta_r = \frac{1}{2}\sum_{i+j=r}[\eta_i,\eta_j]
\end{equation*}
\end{namedtheorem}
Here $\Sym^iE^* \otimes \End(E) \subset \At^k(E)$ is identified with the vertical vector fields of degree $i$. The corresponding bracket is the one induced from the bracket on $\At^k(E)$.

\subsection{Cohomology and symplectic structures}
We now revert back to the case where $W = D$ is codimension one. We can use the description of a $k$-th order foliation as a Maurer-Cartan element to endow $\Omega^{\bullet-1}(D;\bigoplus_{i=0}^k\nu_D^i)$ with a differential. This allows us to compute the cohomology of Lie algebroids of $b^{k+1}$-type:
\begin{namedtheorem}[Corollary \ref{cor:exactsequence}]
Let $A$ be a Lie algebroid of $b^{k+1}$-type. Then there exists an isomorphism
\begin{equation*}
H^{\bullet}(A) \simeq H^{\bullet}(M) \oplus H^{\bullet-1}(D;\bigoplus_{i=0}^k\nu_D^i),
\end{equation*}
where the cohomology of the right-most entry is obtained from a particular $TD$ connection on $\bigoplus_{i=0}^{k}\nu_D^i$ induced by the $k$-th order foliation.
\end{namedtheorem}
As a special case, this recovers Scott's calculation of the cohomology of $b^{k+1}$-tangent bundles. Indeed, in this case the normal bundle $\nu_{D}$ is trivial, and the right hand side reduced to a direct sum of $k+1$ copies of $H^{\bullet-1}(D)$. For general Lie algebroids this cohomology is much more intricate. We show in Proposition \ref{spectralsequenceprop} that there exists a spectral sequence computing this cohomology whose $E_{1}$-page is given by the cohomology of the local systems $\nu_{D}^{i}$ induced by the flat connection on the normal bundle. Furthermore, in Proposition \ref{localalgebraformula} we give a detailed description of the structure of $(\Omega_{A}^{\bullet}, d)$ as a commutative differential graded algebra in the case that $M$ is a tubular neighbourhood of $D$. Interestingly, as we explain in Section \ref{UniversalAlgebrasection}, this algebra structure is controlled by a universal cdga related to the Lie algebra of $k$-jets of vector fields on $\mathbb{R}$.

\subsection{Symplectic structures}
We can now use the description of the Lie algebroid cohomology to study symplectic structures on Lie algebroids of $b^{k+1}$-type. We will provide ample examples of such structures, and present a normal form result around $D$. A $b^{k+1}$-symplectic structure induces a symplectic foliation on $D$. In contrast to the case of $b^{k+1}$-tangent bundles, this symplectic foliation can have non-trivial symplectic variation. 
\begin{namedtheorem}[Proposition \ref{lem:structureonD}]
Let $\omega \in \Omega^2(\A)$ be a $b^{k+1}$-symplectic structure. Then the symplectic foliation on $D$ has symplectic variation given by (the restriction of) the extension class $c(\sigma)\in H^2(D;\nu_D^{-k})$ from Proposition \ref{prop:extensionresult}.
\end{namedtheorem} 



\subsection*{Relation to other work}
Near the time of completing this article, we learned of the preprint \cite{francis2023singular} by Michael Francis, which is based on his 2021 thesis \cite{francis2021groupoids}. In this work, Francis independently introduces the notion of $k$-th order foliations for codimension-one submanifolds, under the name \emph{foliations of transverse order k}, and proves a Riemann-Hilbert correspondence result similar to our Theorem \ref{thm:RH}. However, the focus of our two works seems to be complementary: while we focus on the geometry of the hypersurface Lie algebroids, Francis instead goes on to describe their holonomy groupoids and the resulting groupoid $C^*$-algebras.

\subsection*{Structure of the article}

Section \ref{sec:new} is a completely self-contained section, summarizing the results of this article applicable to the theory of Lie algebroids and providing concrete examples. It is aimed at the reader interested in the applications to the study of singular symplectic structures. Indeed, in this section we describe the geometric structures induced by an algebroid symplectic form on the submanifold $D$ and prove normal form theorems. 

In Section \ref{sec:fols} we set the stage by fixing the terminology regarding holonomic jets of subbundles. This will be used to define the notion of a $k$-th order foliation.  We will also introduce foliations on $k$-th order neighbourhoods, sheaves of invariant functions and show that all these objects are equivalent to singular foliations of type $b^{k+1}$.

In Section \ref{sec:Atiyah} we will introduce the Atiyah algebroid of $k$-jets of vector fields. We will use this to identify $k$-jets of vector fields with splittings of this Atiyah algebroid. We will then use this to introduce a notion of holonomy, which we will employ to obtain a Frobenius-type result.

In Section \ref{sec:RH} we will introduce and prove the Riemann-Hilbert correspondence for $k$-th order foliations. This will be achieved by introducing a universal $J^k\Diff(\rr^l,0)$-bundle and studying its properties. 

In Section \ref{sec:extending} we will introduce and solve the extension problem for $k$-th order foliations. This will be done by using the theory of extension of Lie algebroids, groupoids and groups (which for the readers convenience have been collected in Appendix \ref{sec:appa}). 

In Section \ref{sec:cohomology} we will compute the cohomology of Lie algebroids of $b^{k+1}$-type and study the structure of the resulting cdgas. 

\subsection*{Acknowledgements}
We are grateful to Marius Crainic and Marco Zambon for insightful discussions. We are also grateful to the organizers of the `Higher Geometric Structures along the Lower Rhine XVI' conference, where a preliminary version of this article was presented.

The third author was supported by FWO and FNRS under EOS projects G0H4518N and G012222N, FWO Project G083118N, and directly by the University of Antwerp via a BOF opvangmandaat. 

\section{Symplectic geometry on Lie algebroids of $b^k$-type}\label{sec:new}
In this section we will collect the results in this paper applicable to the study of Lie algebroids. This section is self-contained and aimed at the reader interested in applying the results obtained in this paper to the study of symplectic structures on Lie algebroids.

\subsection{$b^k$-Lie algebroids}
We start with proving the normal form statement of hypersurface Lie algebroids stated in the introduction:
\begin{proof}[Proof of Proposition \ref{prop:localformhypersurfacealgebroid}]
This follows from the Splitting Theorem for Lie algebroids \cite{Duf01, Fer02, Wei00, BLM19}. Namely, let $x \in D$ and let $\gamma: I \to M$ be a curve which passes through $x$ and is transverse to $D$. Then, there is a neighbourhood $U$ of $x$ and an isomorphism of Lie algebroids $A|_U \cong \gamma^{!}A \times TV$, where $V = U \cap D$. Here $\gamma^{!}A$ is a rank-$1$ Lie algebroid over the interval $I$ with the property that its anchor vanishes precisely at the origin to order $k+1$. Hence, $i^{!}A$ is generated by a vector field $x_{1}^{k+1} \partial_{x_{1}}$, for $x_{1}$ a coordinate on $I$. Let $x_{2}, ..., x_{n}$ be coordinates on $V$. Then $A|_U$ has a basis given by the vector fields $x_{1}^{k+1} \partial_{x_{1}}, \partial_{x_{2}}, ..., \partial_{x_{n}}$.
\end{proof}

We will see in Theorem \ref{th:equivalence} that there are several ways to classify Lie algebroids of $b^{k+1}$-type. One way is by associating to them distributions that are involutive up to a certain order:
\begin{namedtheorem}[Definition \ref{def:asssingfol} and Proposition \ref{prop:singfolstojets}]
Let $D \subset M$ be a hypersurface and let $\xi \subset TM$ be a germ of a distribution around $D$. Suppose this distribution is 
\begin{enumerate}
\item tangent to $D$: $\xi|_{D} = TD$, 
\item and is involutive up order $k$, meaning that $j^k_x[X,Y] \in j^k_x\xi$ for all $x \in D$ and $X,Y \in \Gamma(\xi)$ 
\end{enumerate}
Then the sheaf of vector fields which are tangent to $\xi$ along $D$ to order $k$:
\begin{equation}\label{eq:tanfirst}
\Tan^k(M,\xi) := \mathcal{I}_D^{k+1}\mathfrak{X}(M) + \Gamma(\xi),
\end{equation}
defines a Lie algebroid of $b^{k+1}$-type, and every Lie algebroid of $b^{k+1}$-type is of this form.
\end{namedtheorem}
In the above $j^{k}$ denotes taking the $k^{th}$ order jet and $\mathcal{I}_{D}^{k+1}$ denotes the ideal of functions on $M$ vanishing along $D$ to order at least $k+1$. 

We will study this construction in detail in the Section \ref{sec:fols} below. In particular, we show that the Lie algebroid is determined by the $k$-jet $\sigma = j^k_D\xi$ and that there is a bijection between $b^{k+1}$-algebroids and the data of such $k$-jets of distributions.

\begin{remark}\label{rem:clash}
There is an unfortunate clash of terminology between the theory of $b^k$-tangent bundles and jets of vector fields. Informally, Scott states in \cite{MR3523250} that the $b^k$-tangent bundle consists of ``vector fields $k$-tangent to $D$''. In particular, this defines vector fields tangent to $D$ to be called $1$-tangent to $D$.

This is however inconsistent with the point of view of jets: A vector field $X$ being tangent to $D$, means that for any $p \in D$ the value of $X$ at $p$ is in $T_pD$: $X_p \in T_pD$. In terms of jets, this is a condition on the zero jet, and hence should be called 0-tangent.

Consequently, we will say that the $b^{k+1}$-tangent bundle consists of ``vector fields $k$-tangent to $D$''. This ensures that the meaning of the $b^k$-tangent bundle remains unchanged, and only the terminology of $k$-tangency is different from \cite{MR3523250}.
\end{remark}

\subsubsection{Scott's Lie algebroids}
In \cite{MR3523250}, the $b^{k+1}$-tangent bundles are introduced making use of a semi-global defining function $f$ for the hypersurface $D$. Given this defining function one considers the jet $j = j^{k}f$ and the module
\begin{equation*}
\mathfrak{X}(I_D^{k+1},j) :=\set{ X \in \mathfrak{X}^1(M) : \mathcal{L}_X(f) \in I_D^{k+1} \text{ for all } f \in j}.\hfill \qedhere
\end{equation*}
In \cite{MR3523250} it is then proven that this is a well-defined Lie algebroid, i.e. only depending on $j^{k}f$ and not on the whole $f$.

In the language of Equation \eqref{eq:tanfirst} we have $\mathfrak{X}(I_D^{k+1},j) = \Tan^k(M,\ker df)$.

\subsection{Extension problem}
When working on the total space of a real line bundle we may also classify Lie algebroids of $b^{k+1}$-type in terms of explicit Maurer-Cartan elements. Let $L$ be a real line bundle over $D$ and let $\At^{k}(L)$ be the extended Atiyah algebroid 
\[
\At^{k}(L) = \At(L) \oplus \bigoplus_{i = 1}^{k-1}L^{-i}.
\]
This contains the subalgebroid $\bigoplus_{i = 0}^{k-1}L^{-i}$, which is a bundle of Lie algebras whose bracket is given by 
\begin{equation*}
[-,-] : L^{-i} \times L^{-j} \rightarrow L^{-i-j}, \quad [s,t] = (j-i)s\otimes t.
\end{equation*}

\begin{proposition} \label{MaurerCartanAlgebroidProp}
A $b^{k+1}$-algebroid $A$ on the total space of $L$ is equivalent to the data of a bracket preserving splitting $\nabla + \sigma : TD \to \At^{k}(L)$ of the anchor map of $\At^k(L)$. This decomposes into the data of 
\begin{enumerate}
\item A bracket preserving splitting $\nabla : TD \to \At(L)$ corresponding to a flat connection on $L$.
\item An element $\sigma = \sum_{i = 1}^{k-1} \eta_{i} \in \Omega_{D}^{1}(\bigoplus_{i = 1}^{k-1}L^{-i})$ satisfying the following Maurer-Cartan equation: 
\begin{equation} \label{eq:MCfirst}
d^{\nabla}\eta_{r} + \frac{1}{2} \sum_{i + j = r} (j-i) \eta_{i} \wedge \eta_{j} = 0,
\end{equation}
for $r = 1, ..., k-1$. 
\end{enumerate}
\end{proposition}
This result will be proven in Proposition \ref{prop:MCexplicit}. In particular, we will see that a $b^{k+1}$-algebroid for $k\geq 1$ always induces a linear connection on the normal bundle to $D$. In terms of the above proposition this is the connection $\nabla$. 

Given a Lie algebroid of $b^{k+1}$-type $A_k$, we see that we can truncate $(\eta_1,\ldots,\eta_{k-1})$ to a Maurer-Cartan element $(\eta_1,\ldots,\eta_{i-1})$. Therefore, we see that we obtain a sequence of Lie algebroids of $b^{i+1}$-type:
\begin{equation}
A_k \rightarrow A_{k-1} \rightarrow \cdots \rightarrow A_0 = T(-\log D).
\end{equation}
In fact, we will see below in Proposition \ref{prop:restricted fol} that the Lie algebroid $A_k$ may be described as a canonical elementary modification of $A_{k-1}$ at $W$.

The sequence of Lie algebroids leads us to the following problem:
\begin{problem}[The extension problem]
Given a Lie algebroid of $b^{k+1}$-type $A_k \rightarrow M$, does there exist a Lie algebroid of $b^{k+2}$-type $A_{k+1} \rightarrow M$ for which the induced Lie algebroid of $b^{k+1}$-type is $A_k$?
\end{problem}
When this is the case we will say that $A_{k+1}$ \textbf{extends} $A_k$. In terms of the $\eta_i$, we see that the extension problem boils down to finding $\eta_k \in \Omega^1(D;L^{-k})$ such that Equation \ref{eq:MCfirst} is satisfied. Finding this may be addressed purely in cohomological terms:

\begin{theorem} \label{extensionclasssection2}
Let $A_k \rightarrow M$ be a Lie algebroid of $b^{k+1}$-type, with corresponding $(\eta_1,\ldots,\eta_{i-1})$. Then there exists a Lie algebroid of type $b^{k+2}$ extending it if and only if the cohomology class
\begin{equation*}
c = \left[
\sum_{i < k/2} (k-2i)\eta_i \wedge \eta_{k-i}\right] \in H^2(D;L^{-k})
\end{equation*}
vanishes.
\end{theorem}
This cohomology class will be studied in detail in Section \ref{sec:extending}. In particular, we will see that it coincides with the extension class of the sequence 
\[
0 \to L^{-k} \to A_{k}|_{D} \to TD \to 0. 
\]

\begin{example} \label{genusgsurfaceexampleextension}
Let $S$ be a genus $g$ Riemann surface. Its de Rham cohomology ring is given by 
\[
H^{\bullet}(S) \cong \mathbb{R}[\alpha_{i}, \beta_{i}, \omega ; \ i = 1, ..., g]/(\alpha_{i} \alpha_{j} = \beta_{i} \beta_{j} = 0, \alpha_{i} \beta_{j} = \delta_{ij} \omega),
\]
where the degrees are $|\alpha_{i}| = |\beta_{i}| = 1$ and $|\omega| = 2$. Choose representatives for the degree $1$ classes. Abusing notation, we denote these closed $1$-forms by $\alpha_{i}$ and $\beta_{i}$. Consider the trivial line bundle $L = S \times \mathbb{R}$ with trivial connection. Let $t$ be a coordinate on $\mathbb{R}$. By Proposition \ref{MaurerCartanAlgebroidProp}, a $b^{4}$-algebroid can be specified by the data of two closed $1$-forms: 
\[
\eta_{1} = \sum_{i = 1}^{g} (x_{i} \alpha_{i} + y_{i} \beta_{i}), \qquad \eta_{2} = \sum_{i = 1}^{g} (w_{i} \alpha_{i} + z_{i} \beta_{i}),
\] 
where $(x_{i}, y_{i}, w_{i}, z_{i}) \in \mathbb{R}^{4g}$. The associated $b^{4}$-algebroid $A_{3}$ is the subsheaf generated by the following vector fields on the total space of $L$: 
\[
X + \eta_{1}(X) t^2 \partial_{t} + \eta_{2}(X) t^3 \partial_{t}, \qquad t^{4} \partial_{t},
\]
where $X \in \mathfrak{X}(S)$ are vector fields on $S$ which can be lifted to horizontal vector fields on $L = S \times \mathbb{R}$. This algebroid can be extended to a $b^{5}$-algebroid if and only if the cohomology class $c$ from Theorem \ref{extensionclasssection2} vanishes. This cohomology class is given by 
\[
c(x_{i}, y_{i}, w_{i}, z_{i}) = \eta_{1} \wedge \eta_{2} = \sum_{i,j} (x_{i} \alpha_{i} + y_{i} \beta_{i})  (w_{j} \alpha_{j} + z_{j} \beta_{j}) = \sum_{i = 1}^{g} (x_{i} z_{i} - y_{i} w_{i}) \omega. \qedhere 
\]
\end{example}

\begin{example}\label{ex:Heisenberg1}
Let $H(\mathbb{R})$ be the Heisenberg group, which is the subgroup of real valued $3 \times 3$ upper-triangular matrices with $1$'s on the diagonal. Let $H(\mathbb{Z})$ be the subgroup with integer entries. Let $X = H(\mathbb{R})/H(\mathbb{Z})$, which is a compact $3$-dimensional manifold. Writing an upper triangular matrix as 
\[
\begin{pmatrix} 1 & x & z \\ 0 & 1 & y \\ 0 & 0 & 1 \end{pmatrix}
\]
we get local coordinates $(x,y,z)$ on $X$. The $1$-forms $a = dx, b = dy$ and $c = dz - ydx$ are well-defined on $X$ and generate the cohomology. Furthermore, $dc = a \wedge b$. See Example \ref{Heisenbergalgebraexample} for more details on the cohomology. Let $L = X \times \mathbb{R}$ be the trivial line bundle and equip it with the trivial flat connection. By Proposition \ref{MaurerCartanAlgebroidProp}, a $b^{5}$-algebroid $A_{4}$ on $L$ can be specified by the data of three $1$-forms which satisfy the Maurer-Cartan equation. Such a solution is provided by 
\[
\eta_{1} = a, \qquad \eta_{2} = b, \qquad \eta_{3} = -c. 
\]
The cohomology class of Theorem \ref{extensionclasssection2} is given by $-2 a \wedge c$, which is non-trivial in cohomology. Therefore $A_{4}$ cannot be extended to a $b^{6}$-algebroid. Let $t$ be a linear coordinate on $\mathbb{R}$, so that $(x, y, z, t)$ define coordinates on $X$. Then $A_{4}$ has a global basis given in coordinates by 
\[
X_{1} = \partial_{x} + y \partial_{z} + t^2 \partial_{t}, \ \ X_{2} = \partial_{y} + t^3 \partial_{t}, \ \ X_{3} = \partial_{z} - t^4 \partial_{t}, \ \ X_{4} = t^{5} \partial_{t}. \qedhere
\]
\end{example}

\subsection{Cohomology}
Let $A$ be a $b^{k+1}$-algebroid on the total space of a line bundle $L \rightarrow M$, and let $(\eta_1,\ldots,\eta_{k-1})$ be the corresponding Maurer-Cartan element. Define $S_{k}(L) = \oplus_{i = 0}^{k} L^{i}$. In Section \ref{sec:cohomology} we will see that we may endow $S_k(L)$ with a natural flat $TD$-connection induced by the $b^{k+1}$-algebroid. In terms of the Maurer-Cartan element the action of this connection on a section $t_r \in \Gamma(L^r)$ is given by:
\begin{equation}
d(t_{r}) = d^{\nabla}t_{r} + \sum_{i = 1}^{r-1} (i-r) \eta_{i} \otimes t_{r},
\end{equation}
where $\eta_{i} \otimes t_{r} \in \Omega_{D}^1(L^{r-i})$. 

Let $\Omega^{\bullet-1}(D;S_{k}(L))$ be the cochain complex endowed with the differential induced by this connection:
\begin{namedtheorem}[Lemma \ref{decompositionlemma}]
Let $A \rightarrow L$ be a Lie algebroid of $b^{k+1}$-type over a line bundle $L \rightarrow D$. Then there exists $\tau_i \in \Omega^1(A;L^{-i})$, for $0 \leq i \leq k$, such that the following morphism is an isomorphism cochain complexes
\[
\tau : \Omega^{\bullet}_{L} \oplus \Omega^{\bullet - 1}_{D}( S_k(L)) \to \Omega^{\bullet}(A), \qquad (\beta, \sum_{r = 0}^{k} \phi_{r} \otimes u^{(r)}) \mapsto \beta + \sum_{r = 0}^{k} \phi_{r} \wedge \tau_{r}(u^{(r)}).
\]
\end{namedtheorem}
We will also sometimes write $(\phi_r \otimes u^{(r)}) \wedge \tau_r$ for $\phi_r \wedge  \tau_r(u^{(r)})$.

This isomorphism $\tau$ may also be upgraded to an isomorphism of algebras by introducing an algebra structure on $\Omega_{D}^{\bullet-1}(S_k(L))$, which is the content of Proposition \ref{localalgebraformula}. In Section \ref{Heisenbergalgebraexample}, an example of this algebra for a $b^{5}$-algebroid is computed in detail. 

%
%
%
%

We will see that we can use Lemma \ref{decompositionlemma} to compute the cohomology of Lie algebroids of $b^{k+1}$-type on general manifolds:

\begin{namedtheorem}[Proposition \ref{decomp1proof}]
Let $A\rightarrow M$ be a Lie algebroid of $b^{k+1}$-type. Then there exist isomorphisms 
\[
H^{i}(A) \cong H^{i}(M) \oplus H^{i-1}(D, S_{k}(L)),
\]
for all $i$. 
\end{namedtheorem}

\subsection{$b^{k+1}$-symplectic structures}
Now that we understand the cohomology of the Lie algebroids of $b^{k+1}$-type, we can study symplectic structures on them. The first step is to find out what structure is induced on the hypersurface $D$. As our study is completely local around $D$, we will make use of a tubular neighbourhood embedding and assume that $M = L =\nu_D$ for the remainder of this section.

Using the isomorphism $\tau$ in Lemma \ref{decompositionlemma} we immediately obtain that a $b^{k+1}$-symplectic form $\omega \in \Omega^2(A)$ decomposes into a smooth part $\beta \in \Omega^{2}_{L}$ and a singular part $\alpha \in \Omega_{D}^{1} \otimes (L^{0} \oplus ... \oplus L^{k})$. Explicitly:
\begin{lemma}\label{lem:sympdata}
Let $A \rightarrow L$ be a Lie algebroid of $b^{k+1}$-type. Then the data of a closed two-form $\omega \in \Omega^2(A)$ is equivalent to the data:
\begin{itemize}
\item $\beta \in \Omega^{2}_{L}$, with $d\beta = 0$,
\item $\alpha \in \Omega_{D}^{1} \otimes (L^{0} \oplus ... \oplus L^{k})$, satisfying
\end{itemize}
\begin{equation*}
d^{\nabla} \alpha_{p} = p \sum_{r = p+1}^{k} \eta_{r-p} \wedge \alpha_{r},
\end{equation*}
for all $0\leq p \leq k$. This is an iterative system of equations starting with 
\[
d^{\nabla}\alpha_{k} = 0, \ \ d^{\nabla} \alpha_{k-1} = (k-1) \eta_{1} \wedge \alpha_{k}, \ \ d^{\nabla} \alpha_{k-2} = (k-2)\big( \eta_{1} \wedge \alpha_{k-1} + \eta_{2} \wedge \alpha_{k} \big), \ \  ...
\]
\end{lemma}

Now let $\omega$ be a closed algebroid $2$-form. In order for $\omega$ to be non-degenerate in a neighbourhood of the zero section $D$, it suffices to check that the restriction $\omega|_{D}$ to $A|_{D}$ is non-degenerate. Using $\tau$ we can decompose $A|_{D} \cong TD \oplus L^{-k}$ as a vector bundle. The restriction of forms then is described by the following lemma from Section \ref{sec:cohomology}. 

\begin{namedtheorem}[Lemma \ref{lem:restrictmap}]
Under the identifications $\Omega^{\bullet}_{A} \cong \Omega^{\bullet}_{L} \oplus \Omega^{\bullet - 1}_{D}( S_k(L))$ and $\Omega^{\bullet}_{A|_{D}} \cong \Omega^{\bullet}_{D} \oplus \Omega^{\bullet - 1}_{D}(L^{k})$, the restriction of forms $\Omega_{A}^{\bullet} \to  \Omega_{A|_{D}}^{\bullet}$ is given by the following chain map 
\[
 \Omega^{\bullet}_{L} \oplus \Omega^{\bullet - 1}_{D}( S_k(L)) \to  \Omega^{\bullet}_{D} \oplus \Omega^{\bullet - 1}_{D}(L^{k}), \qquad  (\beta, \sum_{r = 0}^{k} \alpha_{r}) \mapsto (\beta|_{D} + \sum_{r = 0}^{k-1} \alpha_{r} \wedge \eta_{r}, \alpha_{k}). 
\]
\end{namedtheorem}


We can now describe non-degeneracy of a $2$-form. 
\begin{lemma} \label{nondegeneracy}
Let $\omega = (\beta, \sum_{r = 0}^{k} \alpha_{r}) \in  \Omega^{2}_{L} \oplus \Omega^{1}_{D}( S_k(L))$ be an algebroid $2$ form. It's restriction to $D$ is given by $\omega|_{D} = (\gamma, \alpha_{k}) \in  \Omega^{2}_{D} \oplus \Omega^{1}_{D}(L^{k})$, where 
\[
\gamma = \beta|_{D} + \sum_{r = 0}^{k-1} \alpha_{r} \wedge \eta_{r}. 
\]
Then $\omega$ is non-degenerate in a neighbourhood of $D$ if and only if $\alpha_{k}$ is nowhere vanishing and $\gamma$ restricts to a non-degenerate $2$-form on the kernel of $\alpha_{k}$. 
\end{lemma}
\begin{proof}
It suffices to determine the conditions for non-degeneracy of $\omega|_{D}$. This in turn is equivalent to $(\omega|_{D})^{n}$ being nowhere vanishing. Writing $\omega|_{D} = \gamma + \alpha_{k}$, and noting that the ranks of $\gamma$ and $\alpha_{k}$ are respectively bounded above by $2n-2$ and $2$, we see that 
\[
(\omega|_{D})^n = n \alpha_{k} \wedge \gamma^{n-1}. 
\]
From this the two conditions follow. 
\end{proof}

Let $\omega \in \Omega^{2}(A)$ be a symplectic form. We now analyze the geometric structures induced on $D$ from the restriction $\omega|_{D} = (\gamma, \alpha_{k})$. Starting with $\alpha_{k} \in \Omega_{D}^{1}(L^{k})$, we note first from Lemma \ref{nondegeneracy} that the kernel $\mathcal{F} = \ker(\alpha_{k}) \subset TD$ defines a corank $1$ distribution on $D$. By Lemma \ref{lem:sympdata}, $\alpha_{k}$ is closed and hence $\mathcal{F}$ defines a foliation. In fact, more is true. 

\begin{proposition} \label{trivialvariationfoliationinduced}
Let $A \to L$ be a $b^{k+1}$-algebroid on the total space of a line bundle over $D$ and let $\nabla$ be the induced flat connection on $L$. Let $\omega \in \Omega^{2}(A)$ be a symplectic form and let $\omega|_{D} = (\gamma, \alpha_{k})$ be its restriction to $D$. Let $\mathcal{F} = \ker(\alpha_{k}) \subset TD$ be the induced foliation on $D$ and let $\nu_{\mathcal{F}}$ be the normal bundle of $\mathcal{F}$, which is equipped with the flat Bott connection $\nabla^{Bott}$. Then $\alpha_{k}$ defines a flat isomorphism 
\[
\alpha_{k} : (\nu_{\mathcal{F}}, \nabla^{Bott}) \to (L^{k}, \nabla^{k})|_{\mathcal{F}}.
\]
In particular, the foliation $\mathcal{F}$ has trivial variation: $\mathrm{var}(\mathcal{F}) = 0$ (see Definition \ref{def:varfol}). 
\end{proposition}
\begin{proof}
By definition of $\alpha_{k}$, it defines an isomorphism of line bundles $\nu_{F} \to L^{k}$, so we must check that it relates the two connections, meaning that for all sections $X \in \mathcal{F}$, $s \in \nu_{F}$, we have 
\begin{equation} \label{flatnessofalpha}
\alpha_{k}(\nabla^{\rm Bott}_{X}(s)) = \nabla_{X}(\alpha_{k}(s)). 
\end{equation}
Let $Y \in \mathfrak{X}(D)$ be a vector field extending $s$, meaning that $p(s) = Y$, where $p: TD \to \nu_{\mathcal{F}}$. By definition of the Bott connection
\[
\nabla^{\rm Bott}_{X}(s) = p([X, Y]). 
\]
Furthermore, since $X$ is tangent to the foliation, $\alpha_{k}(X) = 0$. Hence, Equation \ref{flatnessofalpha} may be rearranged as follows: 
\[
d^{\nabla}(\alpha_{k})(X,Y) = \nabla_{X}(\alpha_{k}(Y)) - \nabla_{Y}(\alpha_{k}(X)) - \alpha_{k}([X,Y]) = 0, 
\]
and this follows because $\alpha_{k}$ is closed. Finally, $\mathcal{F}$ has trivial variation by Proposition \ref{prop:bottextend}. 
\end{proof}

Next, we consider the restriction of $\gamma$ to the foliation, denoted $\omega_{\mathcal{F}} = \gamma|_{\mathcal{F}}$. The next result shows, in particular, that $(\mathcal{F}, \omega_{\mathcal{F}})$ defines a symplectic foliation. This is not surprising. Indeed, $\rho(\omega^{-1}) \in \mathfrak{X}^{2}(L)$ is a Poisson structure and $(\mathcal{F}, \omega_{\mathcal{F}})$ is part of its induced symplectic foliation. What is more surprising is that its symplectic variation is nontrivial and is determined by the extension class of $A|_{D}$. 

\begin{proposition} \label{lem:structureonD}
Let $A \to L$ be a $b^{k+1}$-algebroid on the total space of a line bundle over $D$ and let $c \in H^{2}(D, L^{-k})$ be its extension class, as in Theorem \ref{extensionclasssection2}. Let $\omega \in \Omega^{2}(A)$ be a symplectic form, let $\omega|_{D} = (\gamma, \alpha_{k})$ be its restriction to $D$, and let $\mathcal{F}$ be the induced foliation on $D$ from Proposition \ref{trivialvariationfoliationinduced}. Then 
\[
\omega_{\mathcal{F}} = \gamma|_{\mathcal{F}} \in \Omega^{2}_{\mathcal{F}}
\]
is a foliated symplectic form whose symplectic variation is given by the restriction of $c$: 
\[
\mathrm{var}(\omega_{\mathcal{F}}) = -c|_{\mathcal{F}} \in H^{2}_{D}(\nu^{-1}_{F}). 
\]
\end{proposition}
\begin{proof}
First, we decompose $\omega = (\beta, \sum_{r = 0}^{k} \alpha_{r}) \in \Omega^{2}_{L} \oplus \Omega^{1}_{D}( S_k(L))$. Both $\beta$ and $\sum_{r = 0}^{k} \alpha_{r}$ are individually closed. By Lemma \ref{lem:sympdata} we have $d^{\nabla}\alpha_{k} = 0$. Hence, the closure can be written as 
\[
0 = d\alpha_{k} + d(\sum_{r=0}^{k-1} \alpha_{r}) = \sum_{i = 1}^{k-1} (i-k) \eta_{i} \wedge \alpha_{k} + d(\sum_{r=0}^{k-1} \alpha_{r}). 
\]
This is now an equation in $\Omega^{1}_{D}( S_{k-1}(L))$. Applying the restriction to the $D$ map $\mathcal{R}$ from Lemma \ref{lem:restrictmap} this gives us 
\[
0 = - \sum_{i = 1}^{k-1} (i-k) \eta_{i} \wedge \eta_{k-i} \wedge \alpha_{k} + d \sum_{r = 0}^{k-1} \alpha_{r} \wedge \eta_{r}. 
\]
Using the fact that $\beta$ is closed and the expression for $c$ from Theorem \ref{extensionclasssection2}, this rearranges to the equation $d \gamma = -c \wedge \alpha_{k}$. This implies in particular that $d_{\mathcal{F}}\omega_{\mathcal{F}} = 0$. Hence, since it is non-degenerate by Lemma \ref{nondegeneracy}, it defines a foliated symplectic form. Furthermore, since $\gamma$ is an extension of $\omega_{\mathcal{F}}$ to $\Omega_{D}^2$, it can be used to compute the symplectic variation. Indeed, this is nothing but a connecting homomorphism and is given by 
\[
\mathrm{var}(\omega_{\mathcal{F}}) = \frac{d \gamma}{\alpha_{k}} = -c|_{F}. 
\]
\end{proof}

\subsection{Normal form theorem around the hypersurface}
Using the already abundant theory of normal forms for Lie algebroid symplectic forms we can readily obtain a normal form theorem for $b^{k+1}$-symplectic forms in terms of the data described in Lemma \ref{lem:sympdata}.

\begin{theorem}
Let $\omega \in \Omega^2(A)$ be a $b^k$-symplectic form. Then there exists a tubular neighbourhood $L$ of $D$, on which $\omega$ is symplectomorphic to
\begin{equation*}
\omega'= \sum_{i=0}^k \alpha_i \wedge \tau_i + p^*\iota^*\beta,
\end{equation*}
where the $\alpha_i$ and $\beta$ are as in Lemma \ref{lem:sympdata}, $p : L \to D$ is the projection, and $\iota : D \to L$ is the inclusion of the zero section. 
\end{theorem}
\begin{proof}
By the discussion above we have that $\omega|_D = \omega'|_D$. Therefore, by choosing the tubular neighbourhood small enough, we can ensure that $\omega'$ is non-degenerate. Next note that by construction $\omega-\omega'$ extends smoothly to zero over $D$. Therefore we are in the setting to use Theorem 4.61 from \cite{klaasse2018poisson}. The desired result then follows immediately.
\end{proof}

\subsection{The dual of the Lie algebroid}
Recall that there exists a Poisson structure on the dual of any Lie algebroid. This Poisson structure is itself Lie algebroid symplectic for the right choice of Lie algebroid:

Let $A \rightarrow M$ be a Lie algebroid, and write $\pi : A^* \rightarrow M$ for the dual. Then $A^*$ admits a Poisson structure $\pi_{A^*}$ defined by the brackets:
\begin{equation*}
\{s_1,s_2\}_{A^*} = [s_1,s_2], \quad \{s,f\}_{A^*} = \rho_{A}(s)(f), 
\end{equation*}
where we identify $s_1 \in \Gamma(A)$ with fibre-wise linear functions on ${\rm tot}(A^*)$. It can be shown that this is in fact a Lie algebroid symplectic structure:
\begin{proposition}[\cite{LMM04,SMILDE2022682}]
Let $\lambda_{\rm can} \in \Omega^1(\pi^!A)$ be the canonical one-form defined by
\begin{equation*}
\lambda_{\rm can}(\alpha)(v) = \alpha(d_{A}\pi(v)),
\end{equation*}
for all $\alpha \in A^*$ and $v \in (\pi^!A)_v$. Then $\omega_{\rm can} := -d\lambda_{\rm can}$ is an $\pi^!A$-symplectic form, which induces $\pi_{A^*}$.
\end{proposition}

Therefore, whenever one is given a Lie algebroid one immediately obtains Lie algebroid symplectic structures on its dual. We will now see that the pull-back Lie algebroid of a Lie algebroid of $b^{k+1}$-type is again of $b^{k+1}$-type:
\begin{proposition}\label{prop:tautsymp}
Let $A \rightarrow M$ be a Lie algebroid of $b^{k+1}$-type corresponding to a $k$-th order foliation $\sigma$. Then $\pi^!A \rightarrow A^*$ is a Lie algebroid of $b^{k+1}$-type corresponding to the $k$-th order foliation $\pi^*\sigma$.
\end{proposition}
Therefore, we immediately obtain a large class of examples of symplectic structures on Lie algebroids of $b^{k+1}$-type.

\subsection{Some explicit examples}
In this section we give some explicit examples of $b^{k}$-symplectic forms. 
\subsubsection{Genus $g$ surface}
Let $S$ be a genus $g$ surface. In Example \ref{genusgsurfaceexampleextension} we constructed $b^4$-algebroids on the total space of the trivial line bundle $L = S \times \mathbb{R}$. Now let $D = S \times \mathbb{R}$, with coordinate $u$ on the factor $\mathbb{R}$. Consider the trivial line bundle $L' = D \times \mathbb{R}$ with fibre coordinate $t$. The total space inherits a family of $b^{4}$-algebroid $A_{3}$ depending on $4g$-parameters $(x_{i}, y_{i}, w_{i}, z_{i}) \in \mathbb{R}^{4g}$. Indeed, let $\eta_{1}$ and $\eta_{2}$ be the two closed $1$-forms from Example \ref{genusgsurfaceexampleextension}, now pulled back to $D$. We recall their definitions here: 
\[
\eta_{1} = \sum_{i = 1}^{g} (x_{i} \alpha_{i} + y_{i} \beta_{i}), \qquad \eta_{2} = \sum_{i = 1}^{g} (w_{i} \alpha_{i} + z_{i} \beta_{i}).
\] 
These determine $b^{4}$-algebroids $A_{3}$ just as in Example \ref{genusgsurfaceexampleextension}. Recall that the extension class of the $A_{3}$ is given by the formula $c = \sum_{i = 1}^{g} (x_{i} z_{i} - y_{i} w_{i}) \omega$, where $\omega \in H^{2}(D) \cong \mathbb{R}$ is the generator. 

We now construct a symplectic form. First, let $\beta \in \Omega^2_{S}$ be a symplectic representative of $\omega$, pulled back to a form on $L'$. Define 
\[
\alpha_{1} = -u \eta_{2}, \qquad \alpha_{2} = -2u \eta_{1}, \qquad \alpha_{3} = du. 
\] 
By Lemma \ref{lem:sympdata} these define the data of a closed algebroid $2$-form $\omega \in \Omega^2(A_{3})$. 

Restricting $\omega$ to $D$ we get $(\alpha_{3} = du, \gamma)$, where $\gamma$ is given by
\[
\gamma = \beta - u \eta_{1} \wedge \eta_{2} = \beta - u c. 
\]
The $1$-form $\alpha_{3}$ defines a foliation whose leaves are given by the surface $S$. The restriction of $\gamma$ to this foliation is $\omega_{\mathcal{F}} = \beta - uc$ which is non-degenerate for small values of $u$. Let $D_{\epsilon} = S \times (-\epsilon, \epsilon)$. By Lemma \ref{nondegeneracy}, we therefore obtain a symplectic form on a neighbourhood of $D_{\epsilon}$ for some value $\epsilon > 0$. Furthermore, the symplectic variation is given by $-c \in H^{2}_{\mathcal{F}} \cong H^2_{S}( C^{\infty}( -\epsilon, \epsilon))$. 

Writing the symplectic form out in coordinates we get 
\[
\omega = \beta + (du - ut^2 \eta_{2} - 2u t \eta_{1}) \wedge \frac{dt}{t^4} + (\frac{1}{t^2}\eta_{1} + \frac{1}{t} \eta_{2}) \wedge du. 
\]

\subsubsection{Heisenberg group}
We continue studying the quotient $X = H(\mathbb{R})/H(\mathbb{Z})$ of the Heisenberg group from Example \ref{ex:Heisenberg1}. In terms of the coordinates introduced there, the Lie algebroid $A_4$ has dual basis
\begin{equation*}
e^1 = dx, e^2 = dy, e^3 = dz - ydx, e^4 = \frac{dt}{t^5}- \frac{dx}{t^3}(1+t^2y) - \frac{dy}{t^2} + \frac{dz}{t}
\end{equation*} 
By Proposition \ref{prop:tautsymp} we get that the induced Poisson structure on $Y = {\rm tot}(A_4^*)$ is again a symplectic structure of $b^5$-type. We now explicitly describe this structure. We have that  $Y \simeq D \times \rr^4$, on which we have global coordinates $(x,y,z,t,\alpha,\beta,\eta,\gamma)$. The tautologial form is then given by 
  \[
 \lambda = \alpha e^1 + \beta e^2 + \eta e^3 + \gamma e^4,
 \]
The corresponding symplectic form is then given by
\begin{align*}
\Omega &= d\alpha \wedge e^1 + d\beta \wedge e^2 + d\eta \wedge e^3 + - \eta e^1 \wedge e^2 +  d\gamma \wedge e^4\\
&- \gamma\left( 3t e^4\wedge e^1 + 2t^2e^4 \wedge e^2 + t^3 e^4 \wedge e^3 - 2e^1 \wedge e^3 - te^2 \wedge e^3\right)
\end{align*} 

Along $t=0$ we get a foliation defined by $d\gamma = 0$, and a foliated symplectic form
\begin{equation*}
 \omega = d\alpha \wedge e^1 + d\beta \wedge e^2 - \eta e^1 \wedge e^2 + d\eta \wedge e^3 - 2\gamma e^1 \wedge e^3. 
\end{equation*}
Hence this form has symplectic variation $-2 e^1 \wedge e^3$, which is indeed the extension class as computed in Example \ref{ex:Heisenberg1}, which is consistent with Proposition \ref{lem:structureonD}.

\section{$k$-th order foliations}\label{sec:fols}
We begin (Subsection \ref{ssec:vectorBundles}) by recalling some elementary facts about bundles, subbundles, their sheaves of sections, and their bundles of jets. We do this in order to set notation. We then specialise this discussion to the case of the tangent bundle (Subsection \ref{ssec:distributions}), focusing on the notion of involutivity. This allows us to introduce the Lie algebroids, and some more general singular foliations, to be studied in the paper. Along the way we discuss various equivalent ways of thinking about these objects.

\subsection{Bundles} \label{ssec:vectorBundles}
Let $M$ be a smooth manifold, and let $\Ocal_M$ denote its structure sheaf of smooth functions.  Fix a non-negative integer $k$. Given a fibre bundle $E \rightarrow M$, we write $\Gamma(E)$ for its sheaf of sections and $J^k(E)$ for its bundle of $k$-jets of sections. When $E$ is a trivial fibre bundle with fibre $F$ we simply write $J^k(M,F)$ for the $k$-jets of $F$-valued functions.

When $E = M\times \rr$ is the trivial bundle, we have that $J^k(M,\rr)$ moreover inherits the structure of commutative $\rr$-algebra. It is convenient to recall that the fibre $J^k_p(M,\rr)$ is the stalk at $p \in M$ of the skyscraper sheaf $\Ocal_M/\Ical^{k+1}_p$; here $\Ical^k_p \subset \Ocal_M$ is the ideal of functions that vanish to order at least $k$ at $p$.

More generally, consider a fibre bundle $E \rightarrow M$. Its space of sections $\Gamma(E)$ is an $\Ocal_M$-module. Moreover, $J^k_p(E)$ is the stalk at $p$ of $\Gamma(E)/(\Ical^{k+1}_p\Gamma(E))$. Putting these facts together we deduce that $J^k(E)$ is a bundle of $J^k(M,\rr)$-modules.

\subsubsection{Jets along submanifolds}

Fix a submanifold $i : W \to M$. Write $\Ical_W^k$ for the sheaf of functions vanishing to order $k$ along $W$.
\begin{definition}
The \textbf{$k$-th order neighbourhood of $W$} is defined as:
\[ \Ncal_W^k :=  i^{-1}\dfrac{\Ocal_M}{\Ical_W^{k+1}}, \]
i.e. the sheaf of $k$-jets of functions along $W$.
\end{definition}
In algebraic geometry it is customary to think of $W$ in terms of its sheaf of functions $\Ocal_W$. The restriction map $\Ncal_W^k \rightarrow \Ocal_W$ then allows us to think of $\Ncal_W^k$ as a thickening of $W$, capable of seeing information around $W$ up to $k$-jets.

More generally, consider a fibre bundle bundle $E \rightarrow M$. Recall that a section of $J^k(E)|_W \rightarrow W$ is said to be \textbf{holonomic} if its projection to $J^k(E|_W)$ is the $k$-jet of a section of $E|_W$. 
\begin{definition}
We write $\Hol_W^k(E)$ for the sheaf of holonomic sections of $J^k(E)|_W$. Its elements are said to be \textbf{$k$-jets of sections} of $E$ along $W$.
\end{definition} 
We have maps $\Gamma(E) \rightarrow \Hol_W^k(E) \rightarrow \Gamma(E|_W)$ which we denote by $j^k_W$ and $|_W$, respectively.

The following useful lemma says that holonomic sections have germ representatives:
\begin{lemma} \label{lem:germRepresentative}
Let $E \rightarrow M$ be a fibre bundle. Fix a section $\sigma \in \Hol_W^k(E)$. Then there is a germ of section $s \in \Gamma(E)$ such that $j^k_Ws = \sigma$.
\end{lemma}
\begin{proof} 
Note that
\[ \Hol_W^k(E) = i^{-1}\dfrac{\Gamma(E)}{\Ical_W^{k+1}\Gamma(E)}, \]
hence the statement holds locally. Therefore, by applying a partition of unity we may prove the desired result.
\end{proof}
We can embrace the algebraic viewpoint and study $E$ via its sheaf of sections $\Gamma(E)$. This motivates us to write
\[ E|_{\Ncal_W^k} := \Hol_W^k(E), \]
which is readily seen to be a locally free sheaf of $\Ncal_W^k$-modules.


\subsubsection{Jets of vector subbundles}

Consider now a vector bundle $E \rightarrow M$. We will now study vector subbundles $R \subset E$, which we view as sections of the Grassmannian bundle $\Gr(E,l)$, with $l$ the rank. Recall that there is a correspondence between subbundles $R \subset E$ and locally free subsheaves $\Gamma(R) \subset \Gamma(E)$ of $\Ocal_M$-modules. Our goal now is to translate this to the jet setting. 

By exactness of taking jets, the exact sequence of vector bundles
\[ 0 \rightarrow R \rightarrow E \rightarrow E/R \rightarrow 0 \]
translates into a exact sequence of $J^k(M,\rr)$-modules:
\[ 0  \rightarrow J^k(R) \rightarrow J^k(E) \rightarrow J^k(E/R) \rightarrow 0. \]
An element $\sigma \in J^k(E)$ is \textbf{tangent} to $R$ if it is in the kernel of the second map (or equivalently, in the image of the first). A section $s: M \to E$ is \textbf{$k$-tangent} to $R$ along $W$ if the jets $j^k_Ws$ take values in $J^k(R)$. 

\begin{definition}
We write $\Tangency^k(W,R)$ for the sheaf of sections that are $k$-tangent to $R$ along $W$.
\end{definition}
The following structure theorem provides a rather explicit description of $\Tangency^k(W,R)$:
\begin{lemma} \label{lem:decomposeTangencySheaf}
\begin{align*}
\Tangency^k(W,R) = \Ical_W^{k+1}\Gamma(E) + \Gamma(R).
\end{align*}
Here the symbol $+$ denotes the $\Ocal_M$-span of the union.
\end{lemma}
\begin{proof}
We write $\pi$ for the quotient map $E \rightarrow E/R$. First we observe that $\Tangency^k(W,R)$ contains both $\Ical_W^{k+1}\Gamma(E)$ and $\Gamma(R)$. Indeed, sections $X$ in $\Gamma(R)$ satisfy $j^k\pi \circ j^k(X) = 0$ everywhere, not just along $W$. Similarly, a section $Y$ in $\Ical_W^{k+1}\Gamma(E)$ satisfies $j^rY|_W = 0$, so in particular $j^k\pi \circ j^k(Y)|_W = 0$.

For the other inclusion observe first that we can work locally, since we are arguing with sheaves. Fix a complement $C \subset E$ for $R$. Then, any $X \in \Tangency^k(W,R)$ can be uniquely written as $X = Y + Z$, where $Y \in \Gamma(R)$ and $Z \in \Gamma(C)$. It follows that $Z$ is also in $\Tangency^r(W,R)$. We can then observe that there are canonical isomorphisms of $J^k(M,\mathbb{R})$-modules
\[ J^k(E) \cong J^k(R) \oplus J^k(C) \]
and $j^kZ$ takes values in the term $J^k(C)$. However, $Z$ being in $\Tangency^k(W,R)$ implies that $j^kZ|_W$ takes values in $J^k(R)$. The conclusion is that $j^kZ|_W$ is zero, proving that $Z \in  \Ical_W^{k+1}\Gamma(E)$. This concludes the proof.
\end{proof}

The following characterises jets of subbundles in terms of the $k$-tangency condition:
\begin{lemma} \label{lem:tangenciesSubbundles}
Let $R$ and $R'$ be subbundles of $E$ of rank $l$. The following conditions are equivalent:
\begin{enumerate}
\item $j^k_WR$ and $j^k_WR'$ coincide as sections of the Grassmanian bundle, that is $j^k_WR = j^k_WR' \in \Hol_W^k(\Gr(E,l))$\label{one}
\item $j^k_WR$ and $j^k_WR'$ coincide as subsheaves of $\Hol_W^k(E)$.
\item $\Tangency^k(W,R) = \Tangency^k(W,R')$. \label{three}
\end{enumerate}
\end{lemma}
\begin{proof}
Choose a local frame $f$ for $R$ around $p$. We can provide a chart in $\Gr(E,l)$ around $R$ as follows: Given any $R'$ which is $0$-tangent to $R$ at $p$, then near $p$ $R'$ is graphical over $R$. That is, for all $q$ near $p$ there exists a linear map $A_q \Hom(R_q,R_q^\perp)$ such that $R'$ has as frame $q \mapsto f + A_g(f)$. In these terms, conditions \ref{one} to \ref{three} are all seen to be equivalent to the vanishing of the $k$-jet of $q \mapsto A_q$ at $p$.
\end{proof}

The lemma allows us to discuss jets of subbundles of $E$ along a submanifold $W$.
\begin{definition}
A $k$-jet of subbundle of $E$ along $W$ is an element of $\Hol_W^k(\Gr(E,l))$. The integer $l$ is called the rank of the subbundle. 
\end{definition}
\begin{definition}
A subbundle of $E|_{\Ncal_W^k}$ is a locally free $\Ncal_W^k$-submodule of $E|_{\Ncal_W^k}$. The rank of the subbundle is its rank as a module.
\end{definition}

Given a $k$-jet of subbundle, we can associate to it a subbundle of $E|_{\Ncal_W^k}$ by taking sections:
\begin{lemma} \label{lem:sectionsOfJet}
An element $\sigma \in \Hol_W^k(\Gr(E,l))$ uniquely defines a subbundle $\Gamma(\sigma) \subset E|_{\Ncal_W^k}$.
\end{lemma}
\begin{proof}
We first argue locally. Find a subbundle $R \subset E$ representing $\sigma$. Then $R|_{\Ncal_W^k} \subset E|_{\Ncal_W^k}$ is a subbundle of $E|_{\Ncal_W^k}$. According to Lemma \ref{lem:tangenciesSubbundles}, this subbundle is independent of our choice of $R$. To argue globally, we observe that the choices of $R$ for differing charts are compatible, thanks again to Lemma \ref{lem:tangenciesSubbundles}.
\end{proof}

\begin{lemma} \label{lem:characterisationSubbundles}
The assignment $\sigma \mapsto \Gamma(\sigma)$ is a bijection between:
\begin{itemize}
\item $k$-jets of subbundles of $E$ along $W$.
\item Subbundles of $E|_{\Ncal_W^k}$.
\end{itemize}
\end{lemma}
\begin{proof}
We will provide an inverse. We work locally at first. Given a subbundle $S$ of $E|_{\Ncal_W^k}$, we use the locally free condition to produce a basis of $S$. We then pick a tuple of $l$-linearly independent sections of $E$, close to $W$, representing said basis. The tuple spans a subbundle $R \subset E$ in the vicinity of $W$. All such choices are $k$-tangent to each other along $W$. It follows from Lemma \ref{lem:tangenciesSubbundles} that they are compatible and define a holonomic section of $\Hol_W^k(\Gr(E,l))$. The choices of tuple across charts are compatible, due to Lemma \ref{lem:tangenciesSubbundles}, yielding a global element of $\Hol_W^k(\Gr(E,l))$.
\end{proof}

%
%

\subsection{Distributions} \label{ssec:distributions}

All the results and discussion in Subsection \ref{ssec:vectorBundles} apply to the concrete case of $TM \rightarrow M$ and to its subbundles, the tangent distributions $\xi \subset TM$. The new ingredient in this case is the Lie bracket. Henceforth, we will assume that $\rank(\xi) = {\rm codim}(W) = l$.



\subsubsection{The Lie bracket in $TM|_{\Ncal_W^k}$}
The Lie bracket 
\[ [-,-]: \Gamma(TM) \times \Gamma(TM) \rightarrow \Gamma(TM) \]
is a bidifferential operator of order one. It therefore defines an antisymmetric bundle map
\[ J^k(TM) \times J^k(TM) \rightarrow J^{k-1}(TM) \]
that is compatible with the $J^k(M,\rr)$-module structure on $J^k(TM)$ via the Leibniz identity: i.e. $[fX,Y] = f[X,Y] + df(Y)X$ holds for every $f \in J^k(M,\rr)$ and every pair $X,Y \in J^k(TM)$, all based at the same point.

Due to the loss of a derivative, we do not have a Lie algebra structure on $J^k(TM)$. To remedy this we will restrict ourselves to $\A_W \subset \Gamma(TM)$: the subsheaf of vector fields tangent to $W$. Then $k$-jets of vector fields correspond precisely to the derivations of $\Ncal_W^k$.
\begin{lemma}
The sheaf $i^{-1}(\A_W)/i^{-1}(\Ical_W^{k+1}\Gamma(TM))$ of $k$-jets of vector fields tangent to $W$ is the sheaf of derivations $\Der(\Ncal_W^k)$ of $\Ncal_W^k$. 
\end{lemma}
%
\begin{lemma}
The Lie bracket of vector fields provides $\Der(\Ncal_W^k)$ with the structure of a sheaf of Lie-Rinehart algebras over $\Ncal_W^k$. That is:
\begin{itemize}
\item It is a sheaf of $\Ncal_W^k$-modules.
\item $\Der(\Ncal_W^k)$ acts on $\Ncal_W^k$ by derivations.
\item The Lie bracket in $\Ncal_W^k$ satisfies the Leibniz rule with respect to $\Ncal_W^k$.
\end{itemize}
\end{lemma}
\begin{proof}
Do note that the resulting Lie bracket is precisely the one inherited from being derivations.
\end{proof}
We may now consider
\begin{definition}
A distribution $\xi \subset TM|_{\Ncal_W^k}$ on the $k$-th order neighbourhood $\Ncal_W^k$ is \textbf{involutive} if:
\begin{itemize}
\item It is tangent to $W$, or in other words, it is a subsheaf of $\Der(\Ncal_W^k)$.
\item It is closed under Lie bracket.
\end{itemize}
We will refer to $\xi$ as a \textbf{foliation on the $k$-th order neighbourhood of $W$.}
\end{definition}

\subsubsection{Jets of distributions}
We may now translate the involutivity condition to the corresponding $k$-jet of distribution:
\begin{definition}
We say that a $k$-jet of distribution $\sigma \in \Hol^k_W(\Gr(TM,l))$ is \textbf{involutive at $p \in W$} if for any germ of distribution $\xi$ around $p$ with $\xi_p \in T_pM$ with $j^k_p\xi = \sigma_p$  we have
\begin{itemize}
\item $\xi$ is tangent to $W$ 
\item  $j^k_p[X,Y] \in \sigma_p$, for all $X,Y \in \xi_p$.
\end{itemize}
We say that a $k$-jet distribution is \textbf{involutive} when it is involutive for all $p \in W$.
\end{definition}
We will also refer to an involutive $k$-jet of distribution as a \textbf{$k$-th order foliation}.
\begin{lemma}\label{lem:jetsoffoltofolonjets}
A $k$-jet of distribution $\sigma \in \Hol_W^k(\Gr(TM,l))$ is involutive if and only if the corresponding distribution $\Gamma(\sigma) \subset E|_{\Ncal_W^k}$ is involutive.
\end{lemma}
\begin{proof}
First observe that each local section $X$ of $\Der(\Ncal_W^k)$ defines a $k$-jet of vector field $X_p$ at each $p \in W$. Moreover, $X$ takes values in $\Gamma(\sigma)$ if and only if $X_p$ is $k$-tangent to $\sigma$. 

Suppose then that $\sigma$ is involutive and consider two local sections $X,Y$ of $\Gamma(\sigma)$. The jets $X_p$ and $Y_p$ at $p$ are $k$-tangent to $\sigma$ so, by involutivity, the $(k-1)$-jet $j^k_p[X,Y]$ is again in $\sigma_p$. Reasoning similarly we get that $\Gamma(\sigma)$ being involutive implies that $\sigma$ is involutive.
\end{proof}

\begin{definition}
We say that a distribution $\xi \in TM$ is \textbf{involutive up to order $k$ at $W$} if $\xi$ is tangent to $W$ and $J^k_W[\xi,\xi] \subset J^k_W\xi$.
\end{definition}
By Lemma \ref{lem:germRepresentative}, we have that any holonomic section $\sigma \in \Hol^k_W(\Gr(TM,l))$ can be represented by a germ of distribution $\xi \in TM$ around $W$. The following now follows immediately:
\begin{lemma}
A $k$-jet of distribution $\sigma \in \Hol^k_W(\Gr(TM,l))$ is involutive if and only if there exists a germ $\xi \subset TM$ around $W$ which is involutive up to order $k$ with $J_W^k\xi = \sigma$.
\end{lemma}

\subsection{The associated singular foliation}
We now study the module $\Tan^k(M,\xi)$ introduced in Lemma \ref{lem:decomposeTangencySheaf} for tangent distributions:

\begin{definition}\label{def:asssingfol}
Let $\sigma \in \Hol^k(\Gr(TM,l))$ be $k$-th order foliation, and let $\xi$ be a germ of distribution extending it. We define $\A(\sigma) \subset \mathfrak{X}(M)$ as
\begin{equation*}
\A(\sigma) = \Tan^k(W,\xi).
\end{equation*}
\end{definition}

\begin{lemma}
Let $\sigma \in \Hol^k(\Gr(TM,l))$ be $k$-th order foliation, then $\A(\sigma)$ is a singular foliation with leaves (the connected components) of $M\backslash W$ and $W$. Moreover $\A(\sigma)$ does not depend on the choice of extending distribution $\xi$.
\end{lemma}
\begin{proof}
These facts follow readily from the explicit description of $\Tan^k(W,\xi)$ as in Lemma \ref{lem:decomposeTangencySheaf}.
\end{proof}

We may also express $\A(\sigma)$ in terms of the foliation on the $k$-th order neighbourhood $\Ncal^k_W$. 
\begin{definition}
Let $F \subset \Der(\Ncal^k_W)$ be a foliation, then we define a singular foliation $\A_F \subset \mathfrak{X}(M)$ by
\begin{equation*}
\A_F = p_k^{-1}(i_*F),
\end{equation*}
where $i : W \hookrightarrow M$ is the inclusion and $p_k : T(-\log W) \rightarrow \iota_*\Der(\Ncal^k_W)$ the projection map.
\end{definition}
From this description it is clear that when $F$ and $\sigma \in \Hol^k(\Gr(TM,l))$ are related via Lemma \ref{lem:characterisationSubbundles}, then $\A_F = \A(\sigma)$.

In general, $\A(\sigma)$ is just a singular foliation. We can characterise precisely when it is a Lie algebroid:
\begin{proposition}\label{prop:isliealg}
Let $\sigma \in \Hol^k(\Gr(TM,l))$ be a $k$-th order foliation. Then $\A(\sigma)$ is locally free, and consequently the sections of some Lie algebroid $A(\sigma)$, if and only if $W$ is codimension-one.
\end{proposition}
\begin{proof}
This again follows readily from the explicit description of $\A(\sigma)$ in terms of Lemma \ref{lem:decomposeTangencySheaf}.
\end{proof}

We may characterize precisely what properties the induced singular foliation has.

\begin{definition}\label{def:typek}
Let $M^n$ be a manifold with embedded submanifold $W^{n-l}$. Then a singular foliation $\A$ is said to be of \textbf{$b^{k+1}$-type} if there exists local coordinates $(x_1,\ldots,x_n)$ with $W = \set{x_1=\ldots = x_l}$ and $\A$ given by:
\begin{equation*}
\inp{x^I\partial_{x_1},\ldots,x^I\partial_{x_l},\partial_{x_{l+1}},\ldots,\partial{x_k}},
\end{equation*}
with $I$ running over all multi-indices in $1,\ldots,l$ of length $k+1$.
\end{definition}

\begin{proposition}\label{prop:singfolstojets}
There is an equivalence between $k$-th order foliations and singular foliations of $b^{k+1}$-type. Explicitly:
\begin{itemize}
\item If $\sigma \in \Hol^k(\Gr(TM,l))$ is a $k$-th order foliation, then $\A(\sigma)$ is a singular foliation of type $b^{k+1}$.
\item If $\A$ is of $b^{k+1}$-type, then $\sigma = J^k_W\A$ is a $k$-th order foliation.
\end{itemize}
\end{proposition}
\begin{proof}
The proof of this result depends on a Frobenius theorem for $k$-jets of foliations, whose proof we postpone till Theorem \ref{th:Frobenius}. When $\sigma$ is a $k$-th order foliation, it follows directly from the explicit description of $\A(\sigma)$ from Lemma \ref{lem:decomposeTangencySheaf} and this Frobenius theorem that $\A(\sigma)$ is of $b^{k+1}$-type. Conversely $\A$ being of type $k$ immediately implies that $J^k_W\A$ is a $k$-th order foliation.
\end{proof}

\begin{remark}
There is no a priori reason to restrict ourselves to the case where ${\rm codim}(W) = \rank (\xi)$, the associated singular foliation $\Tan^k(M,\xi)$ is still well-defined. However, it is no longer of $b^{k+1}$-type. We will address singular foliations of this form in upcoming work.
\end{remark}

\subsection{Subsheaves of invariant functions}

Given a foliation $F$ on $\N_W^k$, we may define the following subsheaf $\C_F \subset \N^k_W$ of functions
\begin{equation*}
\C_F = \set{f \in \N_W^k : V(f) = 0 \text{ for all } V \in F}.
\end{equation*}
To characterize the properties of $\C_F$ recall that there are maps $\N^k_W \rightarrow \N_W^1 \rightarrow \mathcal{O}_Z$, and that we have a short exact sequence
\begin{equation*}
0 \rightarrow \nu_Z^* \rightarrow \N_W^1 \rightarrow \mathcal{O}_Z\rightarrow 0.
\end{equation*}
\begin{definition}
We say that $f = (f_1,\ldots,f_l)$ with $f_i \in \N_W^k$ is a \textbf{local defining function} for $W^{n-l} \subset M^n$ if its projection to $\mathcal{O}_W$ is zero, and the corresponding sections $df_i \in \nu_Z^*$ are linearly independent.
\end{definition}

\begin{proposition}
Let $\C_F$ be the sheaf of invariant functions of a foliation $F$. Then $\C_F$ is locally generated over $\rr$ by a local defining function. More precisely, for every $x\in D$, there is a neighbourhood $x\in U \subset M$ and a local defining function $s \in \C_F(U)$ such that we have a surjection of sheaves of algebras
\begin{equation*}
\rr_U[z_1,\ldots,z_l] \rightarrow \C_F(U), \quad f \mapsto f\circ s.
\end{equation*}
Furthermore, this map induces an isomorphism of sheaves $\rr_U[z_1,\ldots,z_l]/\mathcal{I}_k \simeq \C_F(U)$, where $\mathcal{I}_{k+1}$ is the ideal of polynomials vanishing to at least order $k+1$.
\end{proposition}
\begin{proof}
The proof of this result depends on a Frobenius theorem for $k$-jets of foliations, whose proof we postpone till Theorem \ref{th:Frobenius}. By this result, we have that there exists local coordinates $(x_1,\ldots,x_n)$ such that $F = \inp{\partial_{x_{l+1}},\ldots,\partial_{x_n}}$ and $\N^k_W|_U \simeq \mathcal{O}_U[z_1,\ldots,z_l]/(z^I)$, with $I$ running over all multi-indices of length $k+1$. Therefore, we naturally obtain an isomorphism $\varphi : \C_F(U) \simeq \rr_U[x_1,\ldots,x_l]/(z^I)$. The function $s = \varphi^{-1}(x_1,\ldots,x_l)$ provides the desired defining function.
\end{proof}
We may use this to abstractly characterize the sheaf of invariant functions:
\begin{definition}
Let $W^{n-l}$ be an embedded connected submanifold of $M^n$. A \textbf{sheaf of invariant functions} on $\N^k_W$ is
\begin{itemize}
\item a locally constant sheaf $\C_F$ of $\rr$-algebras which is locally isomorphic to $\rr_U[z_1,\ldots,z_l]/\mathcal{I}_{k+1}$, and
\item a morphism of sheaves of $\rr$-algebras $\C_F \rightarrow \N^k$ which sends a set of local generators of $\C_F$ to local defining functions in $\N^k_W$. 
\end{itemize}
\end{definition}
Using this definition we may now recover the foliation:
\begin{lemma}
Let $\C$ be a sheaf of invariant functions on $\N^k_W$. Then the subsheaf of derivations $F_{\C} \subset \Der(\N^k_W)$ which annihilate $\C$ defines a foliation on $\N^k_W$. 
\end{lemma}
\begin{proof}
As by assumption we have that $\C$ is locally isomorphic to $\rr_U[z_1,\ldots,z_l]/(z^I)$, $F_{\C}$ in the same coordinates is generated by $\inp{\partial_{x_{l+1}},\ldots,\partial_{x_n}}$. From this it immediately follows that $F$ is a foliation on $\N^k_W$.
\end{proof}

We conclude:
\begin{proposition}\label{prop:funequiv}
Let $W^{n-l}$ be an embedded connected submanifold of $M^n$. Foliations on $F$ on $\N^k_W$ are in one-to-one correspondence with sheaves of invariant functions on $\N^k_W$.
\end{proposition}

\subsection{Equivalence}
For the reader's convenience we summarize the above results into one result:
\begin{theorem}\label{th:equivalence}
Let $M^n$ be a manifold together with a connected embedded submanifold $W^{n-l}$. Then there is a one-to-one correspondence between:
\begin{enumerate}
\item Singular foliations of type $b^{k+1}$, $\mathcal{A} \subset \mathfrak{X}(M)$ with leaves $W$ and $M\backslash W$.
\item $k$-th order foliations $\sigma \in \Hol^k_W(\Gr(TM,l))$ .
\item Foliations on the $k$-th order neighbrouhood of $W$, $F \subset \Der(\N^k_W)$.
\item Subsheaves of invariant functions $\C \subset \N^k_W$.
\end{enumerate}
\end{theorem}
\begin{proof}
(1)$\leftrightarrow$(2) is the content of Proposition \ref{prop:singfolstojets}. (2)$\leftrightarrow$(3) is the content of Lemma \ref{lem:jetsoffoltofolonjets}.
(2)$\leftrightarrow$(4) is the contents of Proposition \ref{prop:funequiv}.
\end{proof}

\subsection{Remarks about the holomorphic case}
Some of the results discussed in this article can be established equally well in the holomorphic category. In particular, the definitions of singular foliations of type $b^{k+1}$, foliations on the $k$-th order neighbourhood and their equivalence carries over mutatis mutandis. Many of the other results depend on the existence of tubular neighbourhood embeddings and therefore do not generalize to the holomorphic category.


\section{Atiyah algebroid and holonomy}\label{sec:Atiyah}
In this section we will define a principal bundle consisting of $k$-jets of diffeomorphisms, and it's associated Atiyah algebroid (Subsection \ref{ssec:prinbun}). We will see that $k$-th order foliations can be viewed as flat connections on this principle bundle (Subsection \ref{ssec:algsplit}). We will use this point of view to define the holonomy for a $k$-th order foliation (Subsection \ref{ssec:holonomy}, which allows us to proof a Frobenius theorem (Theorem \ref{th:Frobenius}).

\subsection{Principal bundle of $k$-jets}\label{ssec:prinbun}
Let $\pi: E^l \to W$ be a vector bundle. Let $ \mathrm{Diff}_{\pi}(\mathbb{R}^l, E)$, the space of embeddings of $\mathbb{R}^l$ to $E$ with the property that $0$ gets sent to the $0$-section, and $\mathbb{R}^l$ gets sent completely into a fibre of $E$. Then there is a map $\mathrm{Diff}_{\pi}(\mathbb{R}^l, E) \to W$ whose fibre above $x \in W$ is the set of all parametrisations of $E_{x}$. This is a torsor for $\mathrm{Diff}(\mathbb{R}^l, 0)$. Hence, $\mathrm{Diff}_{\pi}(\mathbb{R}^l, E) \to W$ is a right principal $\mathrm{Diff}(\mathbb{R}^l, 0)$-bundle. 

We want to modify this construction by passing to $k$-jets. We first define the group required for this:
\begin{definition}\label{def:Gk}
We let $G_{k,l} := J^k\Diff(\rr^l,0)$ be the group of $k$-jet of diffeomorphisms fixing the origin.
\end{definition}
We may describe $G_{k,l}$ as automorphisms of $k$-jets of functions:
\begin{equation*}
G_{k,l} = \Aut(\rr[z_1,\ldots,z_r]/\mathcal{I}_k),
\end{equation*}
with $\mathcal{I}_k$ the ideal of polynomials vanishing to order at least $k$.

We can now truncate $\mathrm{Diff}_{\pi}(\mathbb{R}^l, E)$ by taking $k$-jets instead
\[
\Fr^k(E) = J^{k}\mathrm{Diff}_{\pi}(\mathbb{R}^l, E). 
\]
This is a right principal $G_{k,l}$-bundle over $W$. 

Note that $\Fr^1(E)$ is precisely the frame bundle ${\rm Fr}(E)$. The Atiyah algebroid of ${\rm Fr}(E)$ is precisely the Atiyah algebroid of $E$. Therefore we now want to explicitly describe the Atiyah algebroid of $\Fr^k(E)$:

\begin{proposition}
Let $E^l \rightarrow D$ be a vector bundle. Then the sections of the Atiyah algebroid $\At^k(E) := \At(\Fr^k(E))$ consist of $k$-jets of vector fields $V$ on $E$ satisfying:
\begin{itemize}
\item $V$ is $\pi$-projectible to a vector field $X$ on $W$,
\item $V$ is tangent to the zero section.
\end{itemize}
\end{proposition}
\begin{proof}
In order to understand $\Gamma(\At^k(E))$, we recall that these are precisely the infinitesimal gauge symmetries (the integration is the gauge groupoid, whose bisections are gauge transformations covering diffeomorphisms of the base). In this case, we see that the gauge symmetries can be realized by ($k$-jets of) diffeomorphisms $\phi: E \to E$ with the properties that 
\begin{itemize}
\item $\phi$ covers a diffeomorphism $g : W \to W$ of the base, 
\item $\phi$ preserves the zero section and $\phi|_{W} = g$. 
\end{itemize}
The infinitesimal version is precisely the desired description.
\end{proof}

It follows from this that $\At^k(E)$ is naturally a subsheaf of $\Der(\mathcal{N}_W^{k})$ consisting of the projectible vector fields. Furthermore, from this description we see that  
\[
\mathcal{N}_W^{k} \At^k(E) = \Der(\mathcal{N}_W^{k}). 
\]

\subsection{Lie algebroid splittings}\label{ssec:algsplit}
Given a vector bundle $E \rightarrow D$, recall that the Atiyah algebroid sits in a short exact sequence:
\begin{equation}\label{eq:Atiyahlin}
0 \rightarrow \End(E) \rightarrow \At(E) \overset{\rho}{\rightarrow} TD \rightarrow 0.
\end{equation}
Vector bundle splittings of \eqref{eq:Atiyahlin} correspond precisely to linear connections on $E$, these connections are flat precisely when the splitting is bracket preserving.

We now consider a non-linear version of \eqref{eq:Atiyahlin} by considering the Atiyah-sequence of $\At^k(E)$:
\begin{equation}\label{eq:Atiyahseq}
0 \rightarrow \ad(\Fr^k(E)) \rightarrow \At^k(E) \rightarrow TW \rightarrow 0
\end{equation}
We can now use this sequence to describe $k$-th order foliations as connections on $\Fr^k(E)$:
\begin{proposition}\label{prop:folsassplittings}
Let $W^{n-l} \subset M$ be an embedded submanifold of codimension $l$, and fix an embedded normal bundle $\nu_W=: E$ for $W$. Then, there is a one-to-one correspondence between $k$-th order foliations and bracket preserving splittings of $\eqref{eq:Atiyahseq}$. Consequently $k$-th order foliations correspond to flat principal $\Fr^k(E)$-connections.
\end{proposition}
\begin{proof}
We view the $k$-th order foliation by its corresponding foliation $F \subset \Der(\mathcal{N}_W^k)$. Viewing $\At^k(E)$ as a subsheaf of $\Der(\mathcal{N}_W^k)$, we may consider $Q := F \cap \At^k(E)$. Then we may define $\sigma : TW \rightarrow \At^k(E)$, by sending a vector $X \in T_xW$ to the unique vector in $Q$ projecting to $W$.
\end{proof}

\subsection{Holonomy}\label{ssec:holonomy}
We may now define the holonomy of a jet of foliation as follows. Just as with usual foliations, the definition of holonomy depends on several choices. We therefore first describe out set-up in some detail:
\subsubsection{Setup}
Our starting data is a connected manifold $W^{n-l} \subset M^n$, a point $x \in W$, and an order $k$. This allows us to consider the class $J^k\Fol_*(W,x,l)$ of triples $(E,\phi,\sigma)$ consisting of:
\begin{itemize}
\item A smooth microbundle $E$ of dimension $\dim(W)+l$ and base $W$. Concretely, $E$ is a germ of vector bundle over $W$. In particular, there is a projection $\pi: E \rightarrow W$ and an inclusion $\iota: W \rightarrow E$, which is a right-inverse of $\pi$.
\item A germ of embedding $\phi: (\RR^l,0) \rightarrow (E,x)$ transverse to $W$, parametrising the fibre $E_x$.
\item A $k$-th order foliation $\sigma$ in $E$ along $W$.
\end{itemize}
We moreover say that two such elements $(E,\phi,\sigma)$ and $(E',\phi',\sigma')$ are isomorphic if there is a germ of diffeomorphism $f: E \rightarrow E'$, fibered over $W$, fixing $W$ pointwise, intertwining $\phi$ and $\phi'$, and pushing forward $\sigma$ to $\sigma'$.

Using the microbundle structure $\pi :E \rightarrow W$ on $E$ we can consider the principal $G_{k,l}$-bundle $\Fr^k(E) = J^k\Diff_{\pi}(\RR^l,E)$. Since $k>0$, we can apply Proposition \ref{prop:folsassplittings} to deduce that $\sigma$ defines a Lie algebroid splitting $\sigma : TW \rightarrow \At^k(E)$. we may now integrate this using Lie II for groupoids \cite{mackenzie2000integration, moerdijk2002integrability} to obtain a holonomy morphism
\begin{equation*}
\hol^\sigma: \Pi_1(W) \rightarrow \Gcal(\Fr^k(E))
\end{equation*}
into the gauge groupoid. This morphism can be restricted to the isotropy at $x$ in order to yield a group homomorphism:
\[ \hol^\sigma_x: \pi_1(W,x) \rightarrow J^k\Diff((E_x,0),(E_x,0)). \]

We then conjugate by $\phi$ in order to yield:
\begin{definition} \label{def:pointedRH}
The \textbf{holonomy} of $(E,\phi,\sigma)\in J^k\Fol_*(W,x,l)$ is defined as:
\begin{align*}
\pi_1(W,x) &\mapsto G_{k,l} \\
[\gamma] &\mapsto \phi^{-1} \circ \hol^\sigma([\gamma]) \circ \phi \hfill \qedhere
\end{align*}
\end{definition}
We will carefully address the dependence on choices of both $E$ and $\phi$, when we address the Riemann-Hilbert correspondence in the next section.

\subsection{Frobenius theorem}
We will prove that $k$-th order foliations, just as foliations, have a simple normal form. But as a $k$-th order foliation might not come from an actual foliation we cannot refer to the standard Frobenius theorem. Instead we will prove:

\begin{proposition}\label{prop:simplyconnected}
Let $\sigma : TW \rightarrow \At^k(E)$ be a k-th order foliation and let $\Sigma$ be its holonomy. When $W$ is simply connected, there exists a germ of actual foliation $\ff$ such that $\sigma = J^k_W\varphi^*\ff$.
\end{proposition}
\begin{proof}
By choosing a tubular neighbourhood we may assume that we are working on the normal bundle $\nu_W \rightarrow W$. As $W$ is simply connected, the holonomy is given by a groupoid morphism $\Sigma : W\times W \rightarrow \mathcal{G}(\Fr^k(E))$. Therefore, for every pair $(x,y)\in W\times W$ we obtain a $k$-jet of diffeomorphism $\varphi_{x,y} : \nu_{W,x} \rightarrow \nu_{W,y}$ given by $\Sigma(x,y)$. Now fix a point $x \in W$, then we obtain a $k$-jet of diffeomorphism
\begin{equation*}
\varphi : \nu_W \rightarrow W \times \nu_{W,x}, v \mapsto \varphi_{x,y}^{-1}(v).
\end{equation*}
As we are working on a vector bundle, we may represent this $k$-jet of diffeomorphism by an actual diffeomorphism $\psi : \nu_W \rightarrow W\times \nu_{W,x}$. For this diffeomorphism it holds that $(\psi^{-1})^*\sigma$ is the $k$-jet of the trivial foliation $\ff$. Hence $\sigma = J^k_W\psi^*\ff$, which finishes the proof.
\end{proof}

Consequently, we can immediately deduce the following local form result for $k$-th order foliations:
\begin{theorem}[Frobenius]\label{th:Frobenius}
Given a foliation $F \in \Der(\Ncal^k)$, then there exists local coordinates $(x_1,\ldots,x_n)$ in which $W = \set{x_1=\ldots = x_l = 0}$ and $F$ is spanned by
\begin{equation*}
\set{\partial_{x_{l+1}},\ldots,\partial_{x_n}}.
\end{equation*}

\end{theorem}

\subsection{Jets of foliations as Maurer-Cartan elements}
For the remainder of this section, suppose we have chosen an embedded normal bundle so that we may assume we are working on a vector bundle $E \rightarrow W$.

If we have two flat splittings $\sigma$ and $\tsigma$ of \eqref{eq:Atiyahseq} we have that their difference is an element in $\Omega^2(D;\ad(\Fr^k(E)))$. By Proposition \ref{prop:algsplittingsabstract} both $\sigma$ and $\tsigma$ endow $\Omega^2(D;\ad(\Fr^k(E)))$ with the structure of a dgla, and $\sigma -\tsigma$ becomes a Maurer-Cartan element.

If one is given a $k$-th order foliation $\sigma_k$, one in particular is given a $1$-th order foliation $\sigma_1$ which corresponds to a linear connection $\nabla$ on $E$. This linear connection in fact induces an actual foliation $\ff^1$ on $E$, and hence a $k$-th order foliation for every $k$. Therefore, in describing $k$-th order foliations it is natural to first fix a linear flat connection on $E$, and then classify all $k$-jets which have that connection as their 1-jet. Then we have:

\begin{proposition}
Let $E \rightarrow W$ be a real line bundle endowed with a linear connection $\nabla$. Then we have the following correspondence:
\begin{itemize}
\item Bracket-preserving splittings $TW \rightarrow \At^k(E)$ with linear part given by $\nabla$.
\item Maurer-Cartan elements of $\Omega^2(D;\ad(\Fr^k(E)))$, equipped with the dgla structure induced by $\nabla$ via  Proposition \ref{prop:algsplittingsabstract}.
\end{itemize}
\end{proposition}
\begin{proof}
This follows immediately from Proposition \ref{prop:algsplittingsabstract} applied to  
\begin{equation*}
0 \rightarrow \ad(\Fr^k(E)) \rightarrow \At^k(E) \rightarrow TD \rightarrow 0. \qedhere
\end{equation*}
\end{proof}
\begin{lemma}
We have that
\begin{equation*}
\ad(\Fr^k(E)) \simeq \Sym^{\leq k}\nu_W^* \otimes \End(\nu_W).
\end{equation*}
Under this isomorphism the dgla structure induced by $\nabla$ via  Proposition \ref{prop:algsplittingsabstract} is given by:

\begin{equation*}
[s_i\otimes A, t_j\otimes B] = j(s_i \cdot A(t_j))\otimes B + i(t_j \cdot B(s_i))\otimes A+ (s_i\cdot t_j) \otimes [A,B],
\end{equation*}
for $s_i \in \Sym^i(\nu_W^*),t_j \in \Sym^j(\nu_W^*),A,B \in \End(\nu_W)$
\begin{equation*}
d(\eta_i) = d^{\nabla}\eta_i,
\end{equation*}
where $\nabla$ also denotes the induced connection on $\Sym^{\leq k}\nu_W^* \otimes \End(\nu_W)$.
\end{lemma}
\begin{proof}
The isomorphism is given by $s_i\otimes A \mapsto \hat{s}_i \otimes a(A)$, with $\hat{s}_i$ the fibrewise polynomial function corresponding with $s_i \in \Gamma(\Sym^i\nu_W^*))$, and $a(A)$ the fibre-wise linear vector field corresponding to $A \in \End(\nu_W)$. Using this explicit isomorphism the desired description of the dgla structures readily follows.
\end{proof}
This allows us to make the Maurer-Cartan equation explicit, concluding to:

\begin{proposition}\label{prop:MCexplicit}
Let $E \rightarrow W$ be a vector bundle with flat linear connection $\nabla$. Bracket preserving splittings $\sigma : TW \rightarrow  \At^k(E)$ correspond to tuples $(\eta_0,\ldots,\eta_{k-1})$ with $\eta_i \in \Omega^1(W;\Sym^i\nu_W^* \otimes \End(\nu_W))$ satisfying the Maurer-Cartan equation
\begin{equation*}
d\eta_r = \frac{1}{2}\sum_{i+j=r}[\eta_i,\eta_j]
\end{equation*}

\end{proposition}
In fact, as we assume that the first jet of $\sigma$ is precisely $\nabla$ we have that $\eta_0 = 0$.

\section{The Riemann-Hilbert correspondence} \label{sec:RH}
%

The Riemann-Hilbert correspondence is the name given to various results relating flat connections to representations of the fundamental group. In this section we will prove some results with this flavour for $k$-th order foliations. This will reduce the classification of $k$-th order foliations, up to isomorphism, to a representation theoretical problem (Theorems \ref{thm:pointedRH} and \ref{thm:RH}). 

\subsection{Setup}

Our Riemann-Hilbert statements come in two flavours, which we call parametrised and unparametrised. Both involve jets of foliations $\sigma$ in a manifold $E$ along a submanifold $W$. The ambient $E$ will always be assumed to be a small thickening of $W$. Where $W$ plays the role of a leaf of $\sigma$. In the parametrised case we additionally endow $E$ with the structure of a fibre bundle over $W$. With this structure in place, we may see the holonomy as a flat structure on the principal $G_{k,l}$-bundle $\Fr^k(E)$.
\subsubsection{The parametrised setting}

\begin{remark}
Observe that the data of a $0$-th order foliation is vacuous. Because of this, we assume $k>0$ for the remainder of the section.
\end{remark}
Recall the definition of $J^k\Fol_*(W,x,l)$ from Subsection \ref{ssec:holonomy}.
\begin{definition} \label{def:pointedRH}
The \textbf{parametrised Riemann-Hilbert map} is defined as:
\begin{align*}
\RH_*: J^k\Fol_*(W,x,l) &\rightarrow \Mor(\pi_1(W,x),G_{k,l}) \\
(E,\phi,\sigma) & \mapsto ([\gamma] \mapsto \phi^{-1} \circ \hol^\sigma([\gamma]) \circ \phi).
\end{align*}
\end{definition}

Then, the Riemann-Hilbert correspondence for $k$-th order foliations, in the parametrised setting, reads:
\begin{theorem} \label{thm:pointedRH}
The map $\RH_*$ defines a bijection
\[ \dfrac{J^k\Fol_*(W,x,l)}{\text{isomorphism}} \longrightarrow \Mor(\pi_1(W,x),G_{k,l}). \]
\end{theorem}
Recall that $W$ is assumed to be connected. We prove Theorem \ref{thm:pointedRH} in Subsections \ref{ssec:surjectivityRH} and \ref{ssec:injectivityRH} below. We will address surjectivity first, then injectivity.

\subsubsection{The unparametrised setting}

The starting data in the unparametrised case is still the tuple $(W,x,l,k)$. The class under study is now $J^k\Fol(W,x,l)$, whose elements are pairs $(E,\sigma)$ with $E$ a germ of manifold of dimension $\dim(W)+l$, along $W$, and $\sigma$ a $k$-th order foliation in $E$ along $W$. An isomorphism in $J^k\Fol(W,x,l)$ is now a diffeomorphism germ mapping the $k$-th order foliations to one another.

Instead of the representations, we consider the \textbf{character variety}
\[ \Char(\pi_1(W,x),G_{k,l}) := \Mor(\pi_1(W,x),G_{k,l})/G_{k,l} \]
i.e. the quotient of the representation variety by the action of $G_{k,l}$ by conjugation. Then:
\begin{definition} \label{def:RH}
The \textbf{(unparametrised) Riemann-Hilbert map} is defined as:
\begin{align*}
\RH: J^k\Fol(W,x,l) &\rightarrow \Char(\pi_1(W,x),G_{k,l}) \\
(E,\sigma) & \mapsto \RH_*(E,\phi,\sigma)/G_{k,l}.
\end{align*}
\end{definition}
More precisely: Given $(E,\sigma)$ we endow $E$ with an auxiliary microbundle structure and we choose some germ of parametrisation $\phi$. This produces an element in $J^k\Fol_*(W,x,l)$, allowing us to apply $\RH_*$. We will prove in Subsection \ref{sssec:wellDefRH} that $\RH$ does not depend on these auxiliary choices and is thus well-defined. Then, the (unparametrised) Riemann-Hilbert correspondence, again for $W$ connected, reads:
\begin{theorem} \label{thm:RH}
The map $\RH$ defines a bijection
\[ \dfrac{J^k\Fol(W,x,l)}{\text{isomorphism}} \longrightarrow \Char(\pi_1(W,x),G_{k,l}). \]
\end{theorem}
Surjectivity will be established in Subsection \ref{ssec:surjectivityRH}, and it will follow immediately from the surjectivity of $\RH_*$. Injectivity is proven in Subsection \ref{sssec:injectivityRH}.

\begin{remark}
One can state and prove a version of Theorem \ref{thm:RH} for germs of foliations; see \cite{CC00}.
\end{remark}

\subsection{Surjectivity of the (parametrised) Riemann-Hilbert map} \label{ssec:surjectivityRH}

We begin by proving the surjectivity of the map $\RH_*$. This will immediately imply the surjectivity of the map $\RH$ and the surjectivity part of Theorems \ref{thm:pointedRH} and Theorem \ref{thm:RH} (i.e. half of the Riemann-Hilbert correspondence, both parametrised and unparametrised).

The proof amounts to showing that any representation in $\Mor(\pi_1(W,x),G_{k,l})$ can be realised as the holonomy of some $k$-th order foliation. This is a two-step process:
\begin{itemize}
\item Given a representation $\rho: \pi_1(W,x) \rightarrow G_{k,l}$, we associate to it a flat principal $G_{k,l}$-bundle $P \rightarrow W$.
\item To each principal $G_{k,l}$-bundle $P$ we associate a vector bundle $\pi: E \rightarrow W$ so that $P = J^k\Diff_{\pi}(\RR^l,E)$.
\end{itemize}
It will then follow that the flat structure on $P$ defines a $k$-th order foliation on $E$ along $W$ whose holonomy, seen from the fibre $E_x$, is precisely $\rho$, concluding the argument.

\subsubsection{First step}
The first step of the argument uses the well-known construction of a flat principal $G$-bundle out of a $\pi_1$-representation. We recall this construction, as we need to refer to it later:

\begin{lemma}\label{lem:abstractBundle}
Let $\rho : \pi_1(W,x) \rightarrow G$ be a representation. Then
\begin{equation*}
P := \dfrac{\widetilde{W} \times G}{\pi_1(W,x)},
\end{equation*}
is a principal $G$-bundle over $W$. Moreover, it comes equipped with a representation of the fundamental groupoid
\begin{equation*}
\hol := \Pi_1(W) \rightarrow \mathcal{G}(P),
\end{equation*}
satisfying $\hol_x = \rho$.
\end{lemma}
\begin{proof}
$P$ is a principal bundle by construction, so we just need to construct $\hol$. Recall the identification
\[ \Pi_1(W) = \dfrac{\widetilde{W} \times \widetilde{W}}{\pi_1(W,x)}. \]
We can then write $[y,z]$ for an element in $\Pi_1(W)$ and $[y,a]$ for an element in $P$; here $y,z \in \widetilde W$ and $a \in G_{k,l}$. It follows that $\hol$ can be uniquely defined from $\rho$ by setting $\hol([y,z]) = ([y,a] \mapsto [z,a])$. Suppose then that $y = [c_x]$ is the class of the constant loop at $x$ and $z$ also represents a loop based at $x$. Then $\hol([y,z])([y,a]) = [z,a] = [y,\rho(z)(a)]$, as desired.
\end{proof}
Later on we will also need the fact that there is preferred identification between $G$ and the fibre $P_x$; it is given by $g \mapsto [[c_x],g]$, where $c_{x}$ is the constant loop at $x$. In particular, there is a marked element in $P_x$, the image $[[c_x],\id]$ of the identity $\id \in G$.

\subsubsection{Second step}

The second part of the argument, showing that every principal $G_{k,l}$-bundle is in fact the bundle of jets of fibrewise parametrisations of some vector bundle, is slightly more involved. To establish it, we note that it is enough if we prove that this is the case for the universal $G_{k,l}$-bundle $EG_{k,l}$, since all others are obtained from it by pullback.

We first introduce a model for $EG_{k,l}$ that is particularly suited to our arguments:
\begin{proposition}
The space
\[ \{ j^k_0f \,\mid\, f: (\RR^l,0) \rightarrow (\RR^{\infty},0) \text{ is an embedding with image in a linear subspace of dimension $l$} \} \]
is contractible and has a free $G_{k,l}$-action given by reparametrisation in the source. In particular, it is a model for $EG_{k,l}$.

Moreover, its quotient by $G_{k,l}$ is the grassmannian of $l$-planes $\Gr(\RR^{\infty},l)$, which is then a model for $BG_{k,l}$.
\end{proposition}
\begin{proof}
Let us write $E$ and $B$ for the models of $EG_{k,l}$ and $BG_{k,l}$ under consideration. Observe that an element $j^k_0f \in E$ defines a linear subspace $A \subset \RR^\infty$; concretely, $A$ is the image of the differential $d_0f$. Since $f^{-1}: (A,0) \rightarrow (\RR^l,0)$ is well-defined as a diffeomorphism germ, postcomposition by $j^k_0f^{-1}$ defines a bijective correspondence between the $G_{k,l}$-orbit of $j^k_0f$ and $G_{k,l}$ itself. This proves that the action is free. Moreover, it shows that the quotient is $B$, the grassmannian of $l$-planes.

It remains to prove contractibility. This follows immediately from the contractibility of the space of embeddings $\RR^l \rightarrow \RR^n$, which in turn follows from the Whitney embedding theorem. A more down-to-earth argument goes instead as follows: We can define a deformation retraction from $E$ onto its subspace $E\GL_l$ using the linear interpolation
\begin{equation*}
f_s = s.d_0f + (1-s)f,
\end{equation*}
for each embedding $f$, which provides an analogous interpolation at the level of $k$-jets. This homotopy is fibered over $B$. Since $E\GL_l$ is contractible, $E$ is contractible as well.
\end{proof}

It follows immediately from the proposition that:
\begin{corollary}
Let $\pi: \gamma \rightarrow \Gr(\RR^{\infty},l)$ be the tautological vector bundle. Then:
\[ EG_{k,l} = J^k\Diff_{\pi}(\RR^l,\gamma). \]
\end{corollary}

%
%
%
%

Moreover, since any principal $G_{k,l}$-bundle is isomorphic to the pullback of the universal one and we can use the classifying map to pull back $\gamma$: 
\begin{corollary} \label{cor:concreteBundle}
Every principal $G_{k,l}$-bundle is of the form $J^k\Diff_{\pi}(\RR^l,E)$ for some vector bundle $\pi: E \rightarrow W$.
\end{corollary}

\subsubsection{End of the surjectivity argument}

Let $\rho : \pi_1(W,x) \rightarrow G_{k,l}$ be a representation. According to Lemma \ref{lem:abstractBundle}, we can use $\rho$ to construct a representation $\hol: \Pi_1(W) \rightarrow \Gcal(P)$ into the gauge groupoid of some principal $G_{k,l}$-bundle $P \rightarrow W$. Moreover, we can assume that $\rho$ is precisely the restriction of $\hol$ to $x$.

By Corollary \ref{cor:concreteBundle}, $P$ is isomorphic to $J^k\Diff_{\pi}(\RR^l,E)$, with $\pi : E \rightarrow W$ a vector bundle of rank $l$. We fix one such isomorphism $\Phi$. As we observed after Lemma \ref{lem:abstractBundle}, there is a preferred element $[[c_x],\id]$ in the fibre $P_x$. The isomorphism $\Phi$ maps it to a preferred jet of embedding $j^k_0\phi \in J^k\Diff_{\pi}(\RR^l,E_x)$.

According to Proposition \ref{prop:folsassplittings}, the differential of $\hol$ defines a Lie algebroid splitting $\sigma : TW \rightarrow \At(P)$. It follows that $\sigma$ defines a $k$-th order foliation whose holonomy at $x$, once we conjugate with the parametrisation $j^k_0\phi$, is given by $\rho$. This means that $\RH_*(E,\phi,\sigma) = \rho$, concluding the proof. \hfill$\qed$

\subsection{Injectivity of the Riemann-Hilbert map} \label{ssec:injectivityRH}

We will now address the injectivity half of Theorems \ref{thm:pointedRH} and \ref{thm:RH}. A particular case of this argument, for simply-connected domains, appeared already in our proof of Frobenius' Theorem \ref{prop:simplyconnected}.

\subsubsection{The parametrised case} \label{sssec:injectivityPointedRH}

Let $(E,\phi,\sigma)$ and $(E',\phi',\sigma')$ be two elements in $J^k\Fol_*(W,x,l)$ with the same image $\rho: \pi_1(W,x) \rightarrow G_{k,l}$ under $\RH_*$. We write $\hol: \Pi_1(W) \rightarrow \Gcal(P)$ and $\hol': \Pi_1(W) \rightarrow \Gcal(P')$ for their holonomies.

We now define a $k$-jet of fibrewise diffeomorphism $j^k\Psi$ from $E$ to $E'$, fixing the zero section $W$. We do this in a fibrewise manner, at each point $y \in W$. We use the formula:
\[ j^k_y\Psi := \hol([\gamma]) \circ j^k\phi' \circ j^k\phi^{-1} \circ \hol(\bar{[\gamma]}), \]
where $\gamma$ is some path from $x$ to $y$. This expression is independent of the choice of $\gamma$, thanks to the assumption
\[ j^k\phi^{-1} \circ \hol_x \circ j^k\phi = j^k(\phi')^{-1} \circ \hol_x' \circ j^k\phi'. \]
By construction, $j^k\Psi$ identifies $\sigma$ with $\sigma'$ and intertwines $j^k\phi$ with $j^k\phi'$.

The proof concludes by choosing a germ of fibrewise diffeomorphism $\Psi: (E,W) \rightarrow (E',W)$ extending $j^k\Psi$. This follows from the fact that each $j^k_y\Psi$ can be extended to a fibered-over-$W$ map from $E$ to $E'$, parametrically in $y$. It is a fibrewise embedding sufficiently close to the zero section. This concludes the injectivity argument and therefore the proof of Theorem \ref{thm:pointedRH}. \hfill$\qed$


\subsubsection{The unparametrised Riemann-Hilbert map is well-defined} \label{sssec:wellDefRH}

Suppose $(E,\phi,\sigma)$ and $(E',\phi',\sigma)$ are two elements in $J^k\Fol_*(W,x,l)$ with the same underlying element in $J^k\Fol(W,x,l)$. In particular, this means that $E$ and $E'$ define the same germ of manifold $A$ along $W$, but the microbundle projections may differ.

We argue in two steps. First, consider $(E,\phi,\sigma)$ and $(E,\phi',\sigma)$. As $\RH_*(E,\phi',\sigma)$ differs from $\RH_*(E,\phi,\sigma)$ by conjugation with $j^k(\phi^{-1} \circ \phi') \in G_{k,l}$ it follows that the two define the same element in the character variety $\Char(\pi_1(W,x),G_{k,l})$, as claimed.
It remains to show that $\RH_*$ does not depend on the choice of microbundle structure. To do this, we need the following result (see for instance \cite{Cer61}):

\begin{lemma}
Let $W \subset M$ be an embedded submanifold, and $A \subset M$ a germ of manifold around $W$. The space of microbundle structures on $A$ is contractible.
\end{lemma}
\begin{proof}
The claim follows if we prove that it is non-empty and convex. The fact that it is non-empty follows from the existence of tubular neighbourhoods, hence we only need to address convexity. Fix an auxiliary Riemannian metric on $W$. Given two microbundle structures $\pi: A \to W$ and $\pi': A \to W$ and a sufficiently small neighbourhood of $W \subset A$, it holds that $\pi(a)$ takes values in a geodesic ball of $\pi'(a)$, for all $a$. I.e. there is a unique geodesic connecting the two. This allows us to uniquely define a linear interpolation between $\pi$ and $\pi'$. Sufficiently close to $W$ this interpolation is through submersions which, in an even smaller neighbourhood, are microbundle structures. 
\end{proof}
As there exists a homotopy of microbundle structures between $E$ and $E'$ we may endow $A \times [0,1]$ with this microbundle structure. We then consider the triple $(A \times [0,1],\phi',\sigma \times [0,1]) \in J^k\Fol_*(W \times [0,1],x,l)$. We can apply to it the map $\RH_*$, which lands in $\Mor(\pi_1(W \times [0,1],x),G_{k,l})$. The latter can be identified with $\Mor(\pi_1(W,x),G_{k,l})$, using that $W \equiv W \times \{0\}$ is a deformation retract of $W \times [0,1]$. From this we deduce that
\[ \RH_*(A \times [0,1],\phi',\sigma \times [0,1]) = \RH_*(E',\phi',\sigma). \]
Consider then the germ of diffeomorphism 
\[ \phi'' =\hol^\sigma(\gamma) \circ \phi' :(\RR^l,0) \rightarrow (A_{(x,1)},(x,1)), \]
where $\gamma$ is the path $s \mapsto (x,s)$. It follows that:
\[ \RH_*(A \times [0,1],\phi',\sigma \times [0,1]) = \RH_*(E,\phi'',\sigma), \]
as claimed. This shows that $\RH$ is well-defined.  \hfill$\qed$

\subsubsection{Injectivity in the unparametrised case} \label{sssec:injectivityRH}

We now prove injectivity in Theorem \ref{thm:RH}. Assume that $\RH(E,\sigma) = \RH(E',\sigma')$. This means that we can choose auxiliary parametrisations $\phi$ and $\phi'$ and auxiliary microbundle structures on $E$ and $E'$ so that $\RH_*(E,\phi,\sigma)$ and $\RH_*(E',\phi',\sigma')$ differ by an element $h$ in $G_{k,l}$. It follows that $\RH_*(E,\phi,\sigma) = \RH_*(E', \phi' \circ h,\sigma')$, so we obtain an isomorphism between the two, by Theorem \ref{thm:pointedRH}. This is in particular an isomorphism between $(E,\sigma)$ and $(E',\sigma')$. The proof of Theorem \ref{thm:RH} is complete.  \hfill$\qed$

\section{Extending to higher order}\label{sec:extending}
One question of fundamental importance in this paper is whether a $k$-th order foliation can be extended to a $k+1$ jet of foliation. That is, given a $k$-th order foliation $\sigma$ does there exists a $(k+1)$-th order foliation $\tilde{\sigma}$ with the same $k$-jet as $\sigma$?
 
In this section we will provide two cohomology classes which are the precise obstruction to extending a $k$-th order foliation. The first class is obtained from Lie algebroid extensions (Subsection \ref{sec:restseq}). The second class is a cohomology class in the group cohomology of the fundamental $\pi_1(W,x)$ (Subsection \ref{ssec:groupext}. Using integrating groupoids and Morita invariance and we will see that these two chomology maps are send to each other via the Van-Est map (Subsection \ref{ssec:groupoidext}).

\subsection{Lie algebroid extensions}\label{sec:restseq}
As the entire discussion in this section is local around $W$, we fix a tubular neighbourhood embedding $\mathcal{U} \rightarrow M$. With this choice in place, we can interchange between $k$-th order foliations and splittings $\sigma : TW \rightarrow \At^k(\nu_W)$ at will.

Consider a $k$-th order foliation $\sigma \in \Hol_W(\Gr(TM,l))$ around $W \subset M$, and the associated singular foliation $\mathcal{A}(\sigma)$. Although $\mathcal{A}(\sigma)$ can be represented by a Lie algebroid only when $W$ has codimension-one, the restriction to $W$ is always a Lie algebroid: 
\begin{proposition}\label{prop:restdesc}
Let $\sigma: TW \rightarrow \At^k(\nu_W)$ be a $k$-th order foliation. Then:
\begin{equation*}
\iota^{-1}_W\A(\sigma)\simeq \Gamma(\At^{k+1}(\nu_W)\oplus_{\At^k(\nu_W)}TW),
\end{equation*}
with fibre-sum taken along $\sigma$.
\end{proposition}
\begin{proof}
Sections of the right-hand side consist of $(k+1)$-jets of vector fields with the property that the $k$-jet is contained in the image of the splitting $\sigma: TW \rightarrow \At^k(\nu_W)$. These are exactly the $(k+1)$-jets of vector fields $k$-tangent to $\sigma$ at $W$, which is exactly $\iota^{-1}_W\mathcal{A}(\sigma)$.
\end{proof}

For the purpose of extending a $k$-th order foliation $\sigma$ to a $(k+1)$-jet, we will only need to work with the Lie algebroid $\At^{k+1}(\nu_W)\oplus_{\At^k(\nu_W)}TW$ rather then the full singular foliation. This Lie algebroid sits in the short exact sequence:

\begin{proposition}
Let $\sigma : TW \rightarrow \At^k(\nu_W)$ be a $k$-th order foliation. Then there is a short exact sequence of Lie algebroids:
\begin{equation}\label{eq:sesofrest}
0 \rightarrow \Sym^k(\nu_W^*)\otimes \End(\nu_W) \rightarrow \At^{k+1}(\nu_W)\oplus_{\At^k(\nu_W)}TW \rightarrow TW \rightarrow 0.
\end{equation}
\end{proposition}
\begin{proof}
We only need to proof that $\ker \rho$ is as described. The kernel of $\rho$ is generated precisely by the $(k+1)$-polynomial vector fields, which can in turn be identified with $\Sym^k(\nu_W^*)\otimes \End(\nu_W)$.
\end{proof}
Note that for $k \geq 1$ the sub Lie algebroid $\Sym^k(\nu_W^*)\otimes \End(\nu_W)$ is in fact Abelian. Consequently, by Proposition \ref{prop:algsplittingsabstract} we obtain a canonical $TW$-representation $\nabla^1$ on $\Sym^k(\nu_W^*)\otimes \End(\nu_W)$. But the underlying $1$-th order foliation also determines a splitting $\sigma_1 : TW \rightarrow \At_1(\nu_W)$ which corresponds to a linear connection on $\nu_W$. These connections are related:
\begin{lemma}
Given a $k$-th order foliation with $k\geq 1$, let $(\Sym^k(\nu_W^*)\otimes \End(\nu_W),\nabla^1)$ be the connection induced by \eqref{eq:sesofrest} and let $(\nu_W,\nabla^2)$ be the connection induced by the corresponding splitting $\sigma : TW \rightarrow \At(P_1)$. Then the induced connection $\nabla^{2,\rm ind}$ on $\Sym^k(\nu_W^*)\otimes \End(\nu_W)$ coincides with $\nabla^1$.  
\end{lemma}
\begin{proof}
Let $\tilde{\sigma} : TW \rightarrow \At^{k+1}(\nu_W)\oplus_{\At^k(\nu_W)} TW$ be any splitting of Equation \eqref{eq:sesofrest}. Then $\nabla^1_X(V) = [\tilde{\sigma}(X),V]$ for all $X \in \Gamma(TW)$ and $V \in \Gamma((\Sym^k(\nu_W^*)\otimes \End(\nu_W))$. However, as $\At^{k+1}(\nu_W)\oplus_{\At^k(\nu_W)} TW$ consist of $(k+1)$-jets of vector fields $[\tilde{\sigma}(X),V] = [\sigma_1(X),V]$. As this is precisely $\nabla_X^{2,\rm ind}(V)$ this finishes the proof.
\end{proof}
Therefore, in what follows we may use either description of the connection on $\Sym^k(\nu_W^*)\otimes \End(\nu_W)$ to endow $\Omega^{\bullet}(W;\Sym^k(\nu_W^*)\otimes \End(\nu_W))$ with a differential. We will also simply right $\nabla$ for the induced flat connection on $\nu_W$, and all it's tensor powers.

\begin{definition}
Let $\sigma :TW \rightarrow \At^k(\nu_W)$ be a $k$-th order foliation with $k\geq 1$. The \textbf{extension class} of $\sigma$, $c(\sigma) \in H^2(W;\Sym^k(\nu_W^*)\otimes \End(\nu_W))$ is defined to be the extension class of \eqref{eq:sesofrest}. That is $c(\sigma) = [\Omega]$ with
\begin{equation*}
\Omega(X_1,X_2) = [\tilde{\sigma}(X_1),\tilde{\sigma}(X_2)] - \tilde{\sigma}([X_1,X_2]),
\end{equation*}
for any splitting $\tilde{\sigma}$ of \eqref{eq:sesofrest} and all $X_1,X_2 \in \mathfrak{X}(W)$.
\end{definition}

\begin{remark}
As we are working with $\At^k(\nu_W)$ we are making use of a tubular neighbourhood embedding, however $c(\sigma)$ is independent of this. Indeed, if we are to replace $\At^{k+1}(\nu_W)\oplus_{\At^k(\nu_W)}TW$ in Equation \ref{eq:sesofrest} with $\iota^{-1}_W\mathcal{A}(\sigma)$, we see that the extension class of this sequence, which is simply $c(\sigma)$ is independent of the choice of tubular neighbourhood embedding.
\end{remark}

\begin{remark}
For $k = 0$, a splitting of Equation \eqref{eq:sesofrest} is precisely the data of a linear connection $\nabla$ on $\nu_W$. The induced connection on $\End(\nu_W)$ obtained from Proposition \ref{prop:algsplittingsabstract} is precisely the connection induced by the linear connection. In this case, the definition of the extension class above depends on the choice of splitting. Indeed, in this case the extension class is precisely the curvature of $\nabla$.
\end{remark}

The extension class provides the precise obstruction of lifting a $k$-th order foliation to a $(k+1)$-th order foliation.

\begin{proposition}\label{prop:extensionresult}
A $k$-th order foliation $\sigma: TW \rightarrow \At^k(\nu_W)$ may be extended to a $(k+1)$-jet if and only if the extension class $c(\sigma) \in H^2(W;\Sym^k(\nu_W^*)\otimes \End(\nu_W))$ of \eqref{eq:sesofrest} vanishes.
\end{proposition}
\begin{proof}
A bracket-preserving splitting of \eqref{eq:sesofrest} is precisely a splitting $\tilde{\sigma} : TW \rightarrow \At^{k+1}(\nu_W)$ with the property that the induced map to $\At^k(\nu_W)$ is $\sigma$. As the extension class measures precisely the whether such a bracket-preserving splitting exists, the conclusion is immediate.
\end{proof}

\subsubsection{Elementary modification}
We now shortly specify to the case where $W = D$ is of codimension-one.

Given a $k$-th order foliation $\sigma_k :TD \rightarrow \At^k(\nu_W)$, consider the associated Lie algebroids $A(\sigma_k)$ and $A(\sigma_{k-1})$. As by Proposition \ref{prop:restdesc} $A(\sigma_k) = \At^k(\nu_W)\oplus_{\At^{k-1}(\nu_W)}TD$ the splitting $\sigma_k$ induced in particular a bracket preserving splitting of
\begin{equation*}
0 \rightarrow (\nu_D^*)^{k-1} \rightarrow A(\sigma_{k-1})_{k-1}|_D \rightarrow TD \rightarrow 0,
\end{equation*}
which we again denote by $\sigma_k$. This realizes $TD$ as a Lie subalgebroid $\sigma_k(TD) \subset A(\sigma_{k-1})|_D$ allowing us to proof the following:
\begin{proposition}\label{prop:restricted fol}
Let $\sigma_k : TW \rightarrow \At^k(\nu_W)$ denote a $k$-th order foliation. Then we can describe $A(\sigma_{k})$ as the \textbf{elementary modification}:
\begin{equation}\label{eq:elmod}
A(\sigma_{k}) = [A(\sigma_{k-1}): \sigma(TW)] := \set{ X \in A(\sigma_{k-1}) : X|_W \in \Gamma(\sigma(W))} . 
\end{equation}
\end{proposition}

\begin{proof}
A local computation will suffice. Therefore we may assume that we have coordinates $(x_1,\ldots,x_n)$ such that $D = \set{x_1 = 0}$ and $\sigma$ is given by the level sets of $x_1$. The associated splitting will simply be given by ${\sigma}(TD) = \set{\partial_{x_2},\ldots,\partial_{x_n}}$. If $X$ is a section of $A(\sigma_{k-1})$, then $X|_{D} = f|_Zx_1^{k}\partial_{x_1} + \cdots$. This is in $\sigma(TD)$ if and only if $f$ vanishes along $D$, and hence $X$ is in fact in $A(\sigma_{k})$.
\end{proof}

\subsection{Group extensions}\label{ssec:groupext}
Another approach in studying the extension problem is by first applying the Riemann-Hilbert correspondence. Therefore, we will first describe the group $G_{k,l}$ in a bit more detail in this section.

We again consider the group $G_{k,l} = J^k_0\Diff(\rr^l,0)$. Note that we have natural homomorphisms $\phi_k : G_{k,l} \rightarrow G_{k-1,l}$ which forget the $k$-th order part. The composition of these maps to $G_{1,l} = \GL(\rr^l)$ will be denoted by
\begin{equation*}
a: G_{k,l} \to G_{1,l} = \GL(\rr^l),
\end{equation*}
We have a short exact sequence of groups:
\begin{equation}\label{eq:groupseq}
0 \rightarrow \Sym^k(\rr^{l,*}) \otimes \End(\rr^l) \rightarrow G_{k+1,l} \rightarrow G_{k,l} \rightarrow 0,
\end{equation}
where $\Sym^k(\rr^{l,*}) \otimes \End(\rr^l)$ is equipped with addition. 
%
Consequently, we obtain a $G_{k,l}$-module structure on $\Sym^k(\rr^{l,*}) \otimes \End(\rr^l)$. Explicitly, this module structure is given by
\[
\psi \ast b = a(\psi)^k b, \quad \psi \in G_{k,l}, b \in \Sym^k(\rr^{l,*}) \otimes \End(\rr^l).
\]

The extension class of \eqref{eq:groupseq} is the cohomology class 
\[
[G_{k+1,l}] = H^{2}(G_{k,l}, \Sym^k(\rr^{l,*}) \otimes \End(\rr^l)). 
\]

By Theorem \ref{thm:RH} we have that $k$-th order foliations correspond to representations of the fundamental group in $G_{k,l}$. Hence suppose that we have such a representation 
\[
\psi : \pi_{1}(W,d) \to G_{k,l}.
\]

The extension problem is equivalent to finding a representation $\tilde{\psi} : \pi_{1}(W,d) \to G_{k+1,l}$ with the property that $\phi_{k+1}(\tilde{\psi}) = \psi$. This can be fully classified in terms of $[G_{k+1,l}]$ and $\psi$ as follows:
\begin{theorem}\label{th:secondextensionclass}
Let $\sigma : TW \rightarrow \At^k(\nu_W)$ be a $k$-th order foliation, with corresponding representation $\psi : \pi_{1}(W,d) \to G_{k,l}$. Then $\sigma$ may be extended to a $(k+1)$-th order foliation if and only if
\begin{equation*}
\psi^*[G_{k+1,l}] \in H^2(\pi_1(W,x),\Sym^k(\rr^{l,*}) \otimes \End(\rr^l))
\end{equation*}
vanishes.
\end{theorem}
\begin{proof}
The problem of lifting this to a representation into $G_{k+1,l}$ is equivalent to splitting the sequence 
\begin{equation}\label{eq:sequencegroups}
1 \to \End(\rr^l) \to G_{k+1,l} \times_{G_{k,l}} \pi_{1}(W,d) \to \pi_{1}(W,d) \to 1.
\end{equation}
The extension class of this sequence is precisely $\psi^*[G_{k+1,l}]$ which finishes the proof.
\end{proof}
Note that here $\Sym^k(\rr^{l,*}) \otimes \End(\rr^l)$ becomes equipped with the module structure induced by 
\[
a \circ \psi : \pi_{1}(W,d) \to GL(\rr^l). 
\]
Geometrically, this representation corresponds to $\nu_{W}$ equipped with a a flat connection.

\subsection{Groupoid extensions}\label{ssec:groupoidext}
We have shown that both extension classes $\psi^{*}[G_{k+1,l}] \in H^{2}(\pi_1(W,d), \Sym^k(\rr^{l,*}) \otimes \End(\rr^l))$ and $c\in H^2(W;\Sym^k(\nu_W^*)\otimes \End(\nu_W))$ are the precise obstruction for extending a $k$-th order foliation to a $(k+1)$-th order foliation. We will now show that the two may in fact be related via passing to the associated integrating groupoids.

%
%
%
%
%
%
Consider a $k$-th order foliation $\sigma:TW \rightarrow \At^k(\nu_W)$ and the associated holonomy $\Psi: \Pi^1(W) \to \mathcal{G}(\Fr^k(\nu_W))$. From the destription of $\A(\sigma)|_W$ in Proposition \ref{prop:restdesc} we immediately obtain:
\begin{lemma}
The source-simply connected integration of the Lie algebroid defining $\A(\sigma)|_W$ is given by 
\[
\mathcal{G} = \mathcal{G}(P_{k+1}) \times_{\Psi} \Pi(W).
\]
\end{lemma}
Therefore we obtain a short exact sequence of groupoids:
\begin{equation}\label{eq:groupoidseq}
1 \to \Sym^k(\nu_W^*)\otimes \End(\nu_W) \to \mathcal{G} \to \Pi(W) \to 1. 
\end{equation}
This sequence of groupoids in turn determines an extension class $[\mathcal{G}] \in H^2(\Pi_1(W);\Sym^k(\nu_W^*) \otimes \End(\nu_W))$.
We may consider the Van-Est map
\begin{equation*}
VE : H^2(\Pi_1(W);\Sym^k(\nu_W^*)\otimes \End(\nu_W)) \rightarrow H^2(W;\Sym^k(\nu_W^*)\otimes \End(\nu_W)).
\end{equation*}
By the proof of Theorem 5 in \cite{cra03} we have that $VE([\mathcal{G}]) = c(\sigma_k) \in H^2(W;\Sym^k(\nu_W^*) \otimes \End(\nu_W))$.

As the groupoid $\mathcal{G}$ is transitive it is Morita equivalent to its restriction to a point $x \in W$, which is the isotropy group 
\[
\mathcal{G}_{x} = G_{k+1,l} \times_{G_{k,l}} \pi_{1}(W,x). 
\]
The corresponding sequence of groups has precisely extension class $\psi^*[G_{k+1,l}]\in H^2(\pi_1(W,x),\Sym^k(\nu_{W,x}^*)\otimes \End(\nu_{W,x}))$ so we conclude:
\begin{proposition}
Let $[\mathcal{G}]$ be the extension class of \eqref{eq:groupoidseq}, and $\psi^*[G_{k+1,l}]\in H^2(\Pi(W),\Sym^k(\nu_W^*)\otimes \End(\nu_W))$ the class form Theorem \ref{th:secondextensionclass}. Then under the isomorphism induced by Morita equivalence
\begin{equation*}
H^{2}(\pi_{1}(W,x), \mathbb{R}) \cong H^{2}(\Pi(W), \Sym^k(\nu_W^*)\otimes \End(\nu_W)),
\end{equation*}
 $\psi^*[G_{k+1,l}]$ is send to $[\mathcal{G}]$. We conclude that the Van-Est map pre-composed with Morita equivalence identifies $\psi^*[G_{k+1,l}]$ with $c(\sigma) \in H^2(W;\Sym^k(\nu_W^*)\otimes \End(\nu_W))$
\end{proposition}
An immediate corollary of the fact that $c(\sigma)$ is in the image of the Van-Est map is that:
\begin{corollary}
The extension class $c \in H^2(W;\Sym^k(\nu_W^*)\otimes \End(\nu_W))$ is aspherical, that is:
\begin{equation*}
\int_{S^2}c(\sigma) = 0,
\end{equation*}
for any embedding of $S^2$ in $W$.
\end{corollary}
Note that this also follows from Proposition \ref{prop:simplyconnected}: Indeed, as $S^2$ is simply connected and the extension class is functional we find that the the extension class vanishes when pulled back to $S^2$.
\section{Lie algebroid cohomology}\label{sec:cohomology}
In this section we specialise to the case of $W = D$ a submanifold of codimension $1$. As we saw in Proposition \ref{prop:isliealg}, the singular foliation $\A(\sigma)$, associated to a $k$-th order foliation $\sigma$, arises as the space of sections of a Lie algebroid $A(\sigma)$. In this section we will compute the cohomology of this Lie algebroid. We will first compute this cohomology in the case where the Lie algebroid is defined on the total space of a line bundle $L \to D$ (Subsection \ref{ssec:cohom}), to afterwards globalize this computation to general manifolds (Subsection \ref{ssec:globalcohom}). 

\subsection{The anti-canonical representation}

Let $\pi: L \to D$ be a line bundle over $D$. Note that $\At^k(L)$  can be described as:


\[
\At^k(L) = \At(L) \oplus \bigoplus_{i = 1}^{k-1}L^{-i}. 
\]
The kernel of the anchor map is $\ad(\Fr^k(L))$, the $k$-jets of vertical vector fields that vanish along $D$. Concretely: 
\[
\ad(\Fr^k(L)) = \bigoplus_{i = 0}^{k-1} L^{-i}. 
\]
The Lie bracket between vertical vector fields is given by 
\[
[ - , - ] : \Gamma(L^{-i}) \times \Gamma(L^{-j}) \to \Gamma(L^{-i - j}), \qquad [t, s] = (i-j) t \otimes s. 
\]
The algebroid $\At(L)$ has a canonical representation on $L$ which we denote $\nabla^{0}$. This induces a representation on all tensor powers, which we will also denote $\nabla^{0}$. The Lie bracket between $\At(L)$ and $L^{-i} \subset \ad(\Fr^k(L))$ is given by these connections: 
\[
[ - , - ] : \Gamma(\At(L)) \times \Gamma(L^{-i}) \to \Gamma(L^{-i}), \qquad [X, t] = \nabla^{0}_{X}(t). 
\]

Since $\At^k(L)$ consists of $k$-jets of vector fields, it naturally acts on the $k^{th}$ order neighbourhood of $D$ in $L$. This is given by
\[
\mathcal{N}_{k} = \bigoplus_{i = 0}^{k} L^{-i}
\]
and is naturally a sheaf of algebras over $D$, obtained as a quotient of the tensor algebra of $L^{-1}$. 
\begin{lemma}
The natural $\At^k(L)$ representation on $\mathcal{N}_{k}$ extends the canonical $At(L)$-representation $\nabla^{0}$ on the tensor powers of $L^{-1}$. This action is given by derivations of the algebra. Explicitely, the subbundle $L^{-i} \subset \ad(\Fr^k(L))$ acts via the following formula:
\[
\nabla^{0} :\Gamma( L^{-i}) \times \Gamma(L^{-j}) \to \Gamma(L^{-i - j}), \qquad \nabla^{0}_{t}(s) = j t \otimes s. 
\]
\end{lemma}

Now let 
\[
S_{k}(L) = \mathcal{N}_{k}^* = \bigoplus_{i = 0}^{k } L^{i},
\]
which is equipped with the dual $\At^k(L)$-flat connection, denoted $\nabla^{AC}$. We refer to this as the \emph{anti-canonical representation}.  We can also decompose this connection into its components. The action of $t \in L^{-i}$ on $v \in L^{j}$ vanishes unless $i \leq j$, in which case it is given by 
\[
\nabla^{AC}_{t}(v) = (i-j) t \otimes v \in L^{j-i}. 
\]

\subsection{Maurer-Cartan element} \label{MCdgmodule}
Now let $\sigma :TD \rightarrow \At^k(L)$ be a $k$-th order foliation. This corresponds to a flat $TD$-connection $\nabla$ on $L$ as well as the data of $1$-forms $(0,\eta_{1}, ..., \eta_{k-1}) \in \Omega^{1}(D;\oplus_{i=0}^{k-1} L^{-i})$ satisfying the Maurer-Cartan equation. The anti-canonical connection pulls back to define a flat $TD$-connection $\sigma^{*}(\nabla^{AC})$ on $S_{k}(L)$. This turns $\Omega_{D}^{\bullet}(S_{k}(L))$ into an $\Omega_{D}^{\bullet}$-dg-module. The differential of a section $t_{r} \in \Gamma(L^{r})$ is given by 
\begin{equation} \label{Anticanonicaldifferential}
d(t_{r}) = d^{\nabla}t_{r} + \sum_{i = 1}^{r-1} (i-r) \eta_{i} \otimes t_{r},
\end{equation}
where $\eta_{i} \otimes t_{r} \in \Omega_{D}^1(L^{r-i})$.

\subsection{Algebroid differential forms}\label{ssec:cohom}
Now consider the $b^{k+1}$-algebroid $A(\sigma)$ on the total space of $L$ which is associated to $\sigma$. Its sections are generated by the vector fields
\[
\sigma(X) = \nabla_{X} + \sum_{i = 1}^{k-1} \eta_{i}(X), 
\]
for $X \in \mathfrak{X}(D)$, as well as the vertical vector fields coming from $L^{-k}$. Here, $\nabla_{X}$ is a linear vector field on $L$ and $\eta_{i}(X) \in L^{-i} = L^{-i -1} \otimes L$ is a vector field of the form 
\[
f(x) z^{i+1} \partial_{z}, 
\]
where $z$ is a vertical coordinate. The sections of $L^{-k} = L^{-k-1}\otimes L$ correspond to vector fields of the form 
\[
z^{k+1} \partial_{z}. 
\]

Now we describe the complex of algebroid differential forms for $A(\sigma)$. First, we use the linear $TD$-connection $\nabla$ to describe the differential forms on $L$. The linear connection defines a direct sum decomposition 
\[
TL \cong \pi^*(TD) \oplus \pi^*(L). 
\]
A pair $(X, s)$ defines the vector field $\nabla_{X} + s$. Dualizing, we get 
\[
T^*L \cong \pi^*(T^*D) \oplus \pi^*(L^{-1}). 
\]
Given a form $\beta \in \Omega^{\bullet}_{D}$, we abuse notation and let $\beta$ denote the pullback $\pi^*(\beta) \in \Omega^{\bullet}_{L}$. In this way, viewing $L^{-i}$ as a subset of the functions on $L \setminus D$, the sections of $\Omega^{\bullet}_{D} \otimes L^{-i}$ may be viewed as differential forms on $L \setminus D$. 

Given a section $s \in L^{-1}$, denote the corresponding vertical $1$-form by $s_{\vertrm} \in \Omega^{1}_{L}$. More generally, we can factor $L^{-i} = L^{-i + 1} \otimes L^{-1}$. Hence, for $s \in L^{-i}$, we write $s_{\vertrm} \in \Omega^{1}_{L}$ when we view the first factor as a function on $L \setminus D$ and the second factor as a vertical $1$-form. 

\begin{lemma} \label{derivative1}
Let $s \in L^{-i}$, viewed as a function on $L \setminus D$. Then 
\[
ds = \nabla(s) + i s_{\vertrm},
\]
where $\nabla(s) \in \Omega^{1}_{D}\otimes L^{-i}$. 
\end{lemma}
\begin{lemma} \label{derivative2}
Let $s_{\vertrm} \in L^{-i} = L^{-i + 1}\otimes L^{-1}$ be a vertical $1$-form. Then 
\[
d(s_{\vertrm}) = (\nabla(s))_{\vertrm} \in \Omega^{1}_{D} \otimes L^{i}. 
\]
\end{lemma}
We now use the data of $\sigma$ to construct algebroid forms. 
\begin{lemma}
Given a section $s \in L^{r}$, for $0 \leq r \leq k$, the following expression defines an algebroid $1$-form with a pole of order $r + 1$ along $D$: 
\[
\tau_{r}(s) = s_{\vertrm} - \sum_{i = 1}^{r-1} \eta_{i} \otimes s \in \Omega^{1}(A(\sigma)). 
\]
In other words, we may view $\tau_{r} \in \Omega^{1}(A(\sigma)) \otimes \pi^{*}L^{-r}$. 
\end{lemma}
\begin{remark}
Note that although $\tau_{r}(s)$ has a pole when viewed as a differential form on $L$, when viewed as an algebroid $1$-form for $A(\sigma)$ it is perfectly smooth. 
\end{remark}
\begin{remark}
Let $z$ be a local flat section of $L^{-1}$, which also serves as a vertical coordinate on $L$. We may locally write $\eta_{i} = a_{i} z^{i}$, for $a_{i} \in \Omega_{D}^{1}$. Then 
\[
\tau_{r}(z^{-r}) = \frac{dz}{z^{r+1}} - \sum_{i = 1}^{r-1} \frac{a_{i}}{z^{r-i}}. \hfill\qedhere
\] 
\end{remark}

Now we can describe the dual space $A(\sigma)^*$. 
\begin{lemma} \label{localbasislemma}
Let $\beta_{1}, ..., \beta_{n}$ be a local basis of $T^*D$ and let $s \in L^{k}$ be a non-vanishing section. Then a local basis of $A(\sigma)^*$ is given by the forms $\beta_{i}$ along with $\tau_k(s)$. 

\end{lemma}
\begin{proof}
Let $X_{i}$ be a basis of vector fields on $D$ which is dual to the forms $\beta_{i}$ and let $t = s^{-1}\in L^{-k}$. Then $\sigma(X_{i})$ along with $t$, viewed as a vertical vector field, form a local basis of $A(\sigma)$. The above one forms are the dual basis. 
\end{proof}

Using the local basis of forms provided by Lemma \ref{localbasislemma}, we can write an arbitrary form $\omega \in \Omega^{\bullet}(A(\sigma))$ as follows:
\begin{equation*}
\omega = \alpha \wedge \tau_{k}(s) + \beta,
\end{equation*}
where both $\alpha$ and $\beta$ are smooth `horizontal' forms (i.e. generated by the $\beta_{i}$). We can expand the form $\alpha$ as 
\begin{equation*}
\alpha = \alpha_{0} + \alpha_{1} + ... + \alpha_{k} + \tilde{\alpha}_{k+1},
\end{equation*}
where $\alpha_{i} \in \Omega_{D}^{\bullet} \otimes L^{-i}$ and $\tilde{\alpha}_{k+1}$ vanishes to order at least $k+1$ along $D$. Hence $\tilde{\alpha}_{k+1} \wedge \tau_{k}(s)$ is a smooth form and so 
\[
\omega = \sum_{i = 0}^{k} \alpha_{i} \wedge \tau_{k}(s) + \gamma, 
\]
with $\gamma = \tilde{\alpha}_{k+1} \wedge \tau_{k}(s) + \beta$ a smooth form on $L$.  Now let $t \in L^{-1}$ be a local basis, so that $\alpha_{i} = \phi_{i} \otimes t^{i}$, with $\phi_{i} \in \Omega^{\bullet}_{D}$. Then 
\[
\alpha_{i} \wedge \tau_{k}(s) = \phi_{i} \wedge \tau_{r}(t^{i}s) - \phi_{i} \wedge \sum_{j = r}^{k-1} \eta_{j} \otimes (t^i s), 
\]
where $r = k-i$ and the second summand is smooth. Hence, 
\[
\omega = \sum_{i = 0}^{k} \phi_{i} \wedge \tau_{k-i}(t^{i}s) + \gamma',
\]
where $\gamma'$ is a smooth form on $L$. 

\begin{lemma} \label{decompositionlemma}
The following morphism is an isomorphism of vector spaces 
\[
\tau : \Omega^{\bullet}_{L} \oplus \bigoplus_{r = 0}^{k}  \Omega^{\bullet - 1}_{D}( L^{r}) \to \Omega^{\bullet}(A(\sigma)), \qquad (\beta, \sum_{r = 0}^{k} \phi_{r} \otimes u^{(r)}) \mapsto \beta + \sum_{r = 0}^{k} \phi_{r} \wedge \tau_{r}(u^{(r)}).
\]
\end{lemma}
\begin{proof}
Showing that $\tau$ is injective is straightforward. We have shown above that $\tau$ is surjective locally on $D$. From this we get global surjectivity using a partition of unity. 
\end{proof}
Observe that the second summand $\bigoplus_{r = 0}^{k}  \Omega^{\bullet - 1}_{D}( L^{r})$ is nothing by $\Omega^{\bullet}_{D}(S_{k}(L))[-1]$, the graded vector space underlying the anticanonical dg module of Section \ref{MCdgmodule} with the degrees shifted by $1$. Then, the map $\tau$ may be upgraded to an isomorphism of cochain complexes:
\begin{lemma}
The morphism $\tau: \Omega_{L}^{\bullet} \oplus \Omega^{\bullet}_{D}(S_{k}(L))[-1] \to \Omega^{\bullet}(A(\sigma))$ is a chain map. 
\end{lemma}
\begin{proof}
It suffices to check that $d \tau_{r}(u) = \tau( d u)$ for $u \in L^r$. Starting from the right hand side, with the expression for $du$ given by Equation \ref{Anticanonicaldifferential}, we get 
\begin{align*} 
\tau( d u) &= \tau_{r}(d^{\nabla}u) + \sum_{i = 1}^{r-1}(i-r)\tau_{r-i}(\eta_{i} \otimes u) \\
&= (d^{\nabla}u + \sum_{i = 1}^{r-1} (i-r) \eta_{i} \otimes u)_{\vertrm} + \sum_{i=1}^{r-1} \eta_{i} \wedge d^{\nabla}u - \sum_{i = 1}^{r-1} \sum_{j = 1}^{r - i - 1} (i-r) \eta_{i} \wedge \eta_{j} \otimes u \\
&= (d^{\nabla}u + \sum_{i = 1}^{r-1} (i-r) \eta_{i} \otimes u)_{\vertrm} + \sum_{i=1}^{r-1} \eta_{i} \wedge d^{\nabla}u - \frac{1}{2} \sum_{s = 1}^{r-1} \sum_{i + j = s} (i - j) \eta_{i} \wedge \eta_{j} \otimes u \\
&= (d^{\nabla}u + \sum_{i = 1}^{r-1} (i-r) \eta_{i} \otimes u)_{\vertrm} + \sum_{i=1}^{r-1} \eta_{i} \wedge d^{\nabla}u - \sum_{s = 1}^{r-1} d^{\nabla}\eta_{s} \otimes u \\ 
&= (d^{\nabla}u + \sum_{i = 1}^{r-1} (i-r) \eta_{i} \otimes u)_{\vertrm} - \sum_{i=1}^{r-1}d^{\nabla}( \eta_{i} \otimes u) \\
&= d \tau_{r}(u). 
\end{align*}
In the fourth equality we have used the Maurer-Cartan equation and in the last equation we have used Lemmas \ref{derivative1} and \ref{derivative2}. 
\end{proof}

\begin{corollary}\label{cor:exactsequence}
There is a canonical isomorphism 
\[
H^{\bullet}(A(\sigma)) \cong H^{\bullet}(D) \oplus H^{\bullet - 1}(D, S_{k}(L)). 
\]
\end{corollary}

We now prove a technical lemma. 
\begin{lemma} \label{technicalexactness}
Let $\alpha = \sum_{r = 0}^{k} \alpha_{r} \in \Omega_{D}^{i-1}(S_{k}(L))$ be a closed element. Let $h \in L^{-2}$ be a metric and let $\gamma \in \Omega_{D}^{1}$ be the closed $1$-form such that $\nabla(h) = -2 \gamma \otimes h$. Then 
\[
\tau(\alpha)|_{L \setminus D} = \alpha_{0} \wedge \gamma + d \eta,
\]
where $\eta = (-1)^{i} (\sum_{s = 1}^{k} s^{-1} \alpha_{s} - \frac{1}{2} \log(h) \alpha_{0})$ is a horizontal form with singularities along $D$. 
\end{lemma}
\begin{proof}
We compute 
\begin{align*}
\tau(\alpha) &= \sum_{r = 0}^{k} \tau_{r}(\alpha_{r}) \\ 
&= \sum_{r = 0}^{k} \alpha_{r, \vertrm} + (-1)^i \sum_{r = 0}^{k} \sum_{j = 1}^{r-1} \eta_{j} \wedge \alpha_{r} \\
&= \sum_{s = 0}^{k} \alpha_{s, \vertrm} + (-1)^{i} \sum_{s = 1}^{k-1}\sum_{r = s+1}^{k} \eta_{r-s} \wedge \alpha_{r} \\
&= \sum_{s=0}^{k} \alpha_{s, \vertrm} + (-1)^{i} \sum_{s = 1}^{k-1} s^{-1} d^{\nabla}\alpha_{s} \\
&= \alpha_{0, \vertrm} + \sum_{s = 1}^{k} s^{-1}(-1)^{i} d\alpha_{s}.
\end{align*}
In the fourth equality we have used Proposition \ref{prop:MCexplicit} and in the last equality we have used Lemma \ref{derivative1} to show that 
\[
(-1)^{i}d\alpha_{s} = s \alpha_{s, \vertrm} + (-1)^{i} d^{\nabla}\alpha_{s}. 
\]
Now $\alpha_{0} \in \Omega_{D}^{i-1}$ and so $\alpha_{0, \vertrm} = \alpha_{0} \wedge \epsilon$, where $\epsilon = 1_{\vertrm}$, viewing $1 \in L^{1}\otimes L^{-1}$. This can be constructed from the metric $h$ as $\epsilon = \frac{h_{\vertrm}}{h}$. Now using Lemma \ref{derivative1} we compute 
\[
d\log(h) = \frac{ \nabla(h) + 2 h_{\vertrm}}{h} = -2 \gamma + 2 \epsilon. 
\]
Hence $\epsilon = \gamma + \frac{1}{2} d\log h$, and so 
\[
\alpha_{0, \vertrm} = \alpha_{0} \wedge \gamma - \frac{1}{2} (-1)^{i} d( \log h \alpha_{0}). \qedhere
\]
\end{proof}

Notice that $\Omega_{D}^{\bullet} \oplus \Omega^{\bullet}_{D}(S_{k}(L))[-1]$ is a quasi-isomorphic subcomplex of $\Omega^{\bullet}(A(\sigma))$. 
\begin{proposition} \label{localalgebraformula}
The (isomorphic) image of $\Omega_{D}^{\bullet} \oplus \Omega^{\bullet}_{D}(S_{k}(L))[-1]$ under $\tau$ is a quasi-isomorphic subalgebra of $\Omega^{\bullet}(A(\sigma))$. With respect to this algebra structure, $\Omega_{D}^{\bullet}$ is a subalgebra with its usual structure and $\Omega^{\bullet}_{D}(S_{k}(L))[-1]$ is an ideal. Hence, $\Omega^{\bullet}_{D}(S_{k}(L))[-1]$ is an $\Omega_{D}^{\bullet}$-cdga. The dg-module structure is the usual one, and the product is given by the following formula:
\[
t_{s} \ast t_{r} = \sum_{i = s}^{r-1} \eta_{i} \otimes t_{s} \otimes t_{r}. 
\]
where $s < r$ and $t_{r} \in L^{r}$ and $t_{s} \in L^{s}$. 
\end{proposition}
\begin{proof}
Consider two elements $(\beta_{1}, \sum_{r = 0}^{k} \phi_{r} \otimes u_{r})$ and $(\beta_{2}, \sum_{s = 0}^{k} \psi_{s} \otimes v_{s})$. Computing their product, we get 
\begin{align*}
\tau(\beta_{1}, \sum_{r = 0}^{k} \phi_{r} \otimes u_{r}) \wedge \tau(\beta_{2}, \sum_{s = 0}^{k} \psi_{s} \otimes v_{s}) &= (\beta_{1} + \sum_{r = 0}^{k} \phi_{r} \wedge \tau_{r}(u_{r})) \wedge (\beta_{2} + \sum_{s = 0}^{k} \psi_{s} \wedge \tau_{s}(v_{s})) \\ 
&= \beta_{1} \wedge \beta_{2} + \sum_{s = 0}^{k} \beta_{1} \wedge \psi_{s} \wedge \tau_{s}(v_{s}) + (-1)^{|\beta_{2}|} \sum_{r=0}^{k} \phi_{r} \wedge \beta_{2} \wedge \tau_{r}(u_{r}) \\
& \qquad \qquad + \sum_{r,s=0}^{k} (-1)^{|\psi_{s}|} \phi_{r} \wedge \psi_{s} \wedge \tau_{r}(u_{r}) \wedge \tau_{s}(v_{s}) \\
&= \tau(\beta_{1} \wedge \beta_{2}, \sum_{s = 0}^{k} \beta_{1} \wedge \psi_{s} \otimes v_{s} + (-1)^{|\beta_{2}|} \sum_{r = 0}^{k} \phi_{r} \wedge \beta_{2} \otimes u_{r}) \\ 
& \qquad \qquad + \sum_{r,s=0}^{k} (-1)^{|\psi_{s}|} \phi_{r} \wedge \psi_{s} \wedge \tau_{r}(u_{r}) \wedge \tau_{s}(v_{s}). 
\end{align*}
It thus remains to compute the formula for $\tau_{s}(v_{s}) \wedge \tau_{r}(u_{r})$. We can assume that $s < r$ (otherwise, just flip the order at the cost of a sign). Computing this we get 
\begin{align*}
\tau_{s}(v_{s}) \wedge \tau_{r}(u_{r}) &= (v_{s, \vertrm} - \sum_{i = 1}^{s-1} \eta_{i} \otimes v_{s}) \wedge (u_{r,\vertrm} - \sum_{j = 1}^{r-1} \eta_{j} \otimes u_{r}) \\ 
&= \sum_{j = 1}^{r-1} \eta_{j} \wedge (v_{s} \otimes u_{r})_{\vertrm} - \sum_{i = 1}^{s-1} \eta_{i} \wedge (v_{s} \otimes u_{r})_{\vertrm} + \sum_{j = 1}^{r-1} \sum_{i = 1}^{s-1} \eta_{i} \wedge \eta_{j} \otimes (v_{s} \otimes u_{r}) \\
&= \sum_{j = s}^{r-1} \eta_{j} \wedge (v_{s} \otimes u_{r})_{\vertrm} - \sum_{j = s}^{r-1} \sum_{i = 1}^{s-1} \eta_{j} \wedge \eta_{i} \otimes (v_{s} \otimes u_{r}) \\
&= \sum_{j = s}^{r-1} \eta_{j} \wedge ( (v_{s} \otimes u_{r})_{\vertrm} - \sum_{i = 1}^{s + r - j - 1} \eta_{i} \otimes (v_{s} \otimes u_{r})) + \sum_{\substack{i, j \geq s, \\  i + j \leq r + s -1}} \eta_{j} \wedge \eta_{i} \otimes (v_{s} \otimes u_{r}) \\
&= \tau( \sum_{j = s}^{r-1} \eta_{j} \otimes v_{s} \otimes u_{r}). 
\end{align*}
Here we have used $\eta_{j} \otimes v_{s} \otimes u_{r} \in \Omega_{D}^{1} \otimes L^{s + r - j}$ and the fact that symmetric sums of $\eta_{j} \wedge \eta_{i}$ vanish. 
\end{proof}

\subsection{A spectral sequence}
Note that $\Omega_{D}^{\bullet}(S_{k-1}(L))[-1] \subset \Omega_{D}^{\bullet}(S_{k}(L))[-1] $ is a subcomplex for all $k$. This gives rise to the following bounded decreasing filtration $F^{\bullet}$ on $\Omega_{D}^{\bullet}(S_{k}(L))[-1] $:
\[
F^{p}\big(\Omega_{D}^{\bullet}(S_{k}(L))[-1]\big) =\begin{cases}
   \Omega_{D}^{\bullet}(S_{k}(L))[-1]   & p \leq -k  \\
  \Omega_{D}^{\bullet}(S_{-p}(L))[-1]  & -k \leq p \leq 0 \\
  0 & p>0
\end{cases}
\]
\begin{proposition} \label{spectralsequenceprop}
The filtration $F^{\bullet}$ gives rise to a spectral sequence $E$ which converges to the cohomology $H^{\bullet-1}_{D}(S_{k}(L))$. The $E_{1}$ page is given by the cohomology of the local systems $L^{i}$: 
\[
E_{1}^{p,q} =  H^{p+q-1}(L^{-p}), \qquad p = 0, ..., -k.
\]
\end{proposition}
\begin{proof}
Because the filtration is bounded, the spectral sequence converges to the cohomology. The $E_{0}$ page of the spectral sequence is given by the associated graded of the complex. Hence, for $p = 0, ..., -k$ we have
\[
E_{0}^{p,q} = \frac{F^{p}\big(\Omega_{D}^{p+q}(S_{k}(L))[-1]\big)}{F^{p+1}\big(\Omega_{D}^{p+q}(S_{k}(L))[-1]\big)} \cong \Omega_{D}^{p+q-1}(L^{-p}),
\]
and $E_{0}^{p,q} = 0$ for all other values of $p$. The differential on the complex is given by Equation \ref{Anticanonicaldifferential}. The portion that survives after taking the associated graded is given by $d^{\nabla}$, the connection for the local systems $L^{i}$. Hence, the $E_{1}$ page is given by 
\[
E_{1}^{p,q} = H_{D}^{p+q-1}(L^{-p})
\]
for $p = 0, ..., -k$ and $E_{1}^{p,q} = 0$ for all other values of $p$. 
\end{proof}

\subsection{The algebra in the case of the trivial line bundle}
We study the special case where $L$ is the trivial bundle with trivial local system. Let $t_{r} \in L^{r}$ be the unit section. Hence $t_{r} \otimes t_{s} \cong t_{r+s}$. Now we prove a technical lemma that will be useful in computing examples. 

Let $(C^{\bullet},d) \subset (\Omega^{\bullet}_{D}, d)$ be a quasi-isomorphic sub-cdga. Assume that the Maurer-Cartan elements have the form $\eta_{i} = \alpha_{i} \otimes t_{-i}$ for $\alpha_{i} \in C^1$. Define 
\[
C^{\bullet}(\sigma) = \bigoplus_{i = 0}^{k}(C^{\bullet} \otimes t_{i}[-1]),
\]
where we now view $t_{i}$ to have degree $+1$. It is straightforward to check that $C^{\bullet}(\sigma)$ is a sub-cdga of $\Omega^{\bullet}_{D}(S_{k}(L))[-1]$. 
\begin{lemma} \label{inducedquasiiso}
The cdga $C^{\bullet}(\sigma) $ is a quasi-isomorphic subalgebra of $\Omega^{\bullet}_{D}(S_{k}(L))[-1]$. 
\end{lemma}
\begin{proof}
It suffices to show that the inclusion, denoted $i$, is a quasi-isomorphism. Note first that the filtration $F^{\bullet}$ induces a filtration on $C^{\bullet}(\sigma)$ and hence the inclusion induces a morphism between the associated spectral sequences. The map on the $E_{1}$-page is given by 
\[
i: H^{p+q-1}(C^{\bullet}) \to H^{p+q-1}_{D}.
\]
This is an isomorphism because we have assumed that $C^{\bullet}$ is a quasi-isomorphic subcomplex of $\Omega_{D}^{\bullet}$. Hence, by the Eilenberg-Moore Comparison theorem \cite{MR1269324}, 
$i: C^{\bullet}(\sigma)  \to \Omega^{\bullet}_{D}(S_{k}(L))[-1]$ is a quasi-isomorphism. 
\end{proof}

\subsubsection{An Example} \label{Heisenbergalgebraexample}
We continue with the example of Heisenberg group from Example \ref{ex:Heisenberg1}. So recall that $H(\mathbb{R})$ denotes the Heisenberg group, $H(\mathbb{Z})$ the subgroup consisting of matrices with integer entries and $X = H(\mathbb{R})/H(\mathbb{Z})$ the quotient. Let $a = dx$, $b = dy$ and $c = dz - y dx$. These descend to $1$-forms in $\Omega^{1}(X)$ and satisfy $da = db = 0$ and $dc = a \wedge b$. Let 
\[
C = \mathbb{R}[a,b,c]
\]
be the induced subalgebra of $\Omega^{\bullet}(X)$. Then by the Gysin sequence, $C$ is quasi-isomorphic to $\Omega^{\bullet}(X)$. The cohomology is given by 
\[
H^{\bullet}(C) = \langle 1 \rangle \oplus \langle a ,b \rangle \oplus \langle ac, bc \rangle \oplus \langle abc \rangle.
\]
Choose Maurer-Cartan elements given by 
\[
\eta_{1} = a \otimes t_{-1}, \qquad \eta_{2} = b\otimes t_{-2}, \qquad \eta_{3} = -c \otimes t_{-3}. 
\]
This gives rise to the cdga 
\[
C(\sigma) = \mathbb{R}[a,b,c]\langle t_{0}, t_{1}, t_{2}, t_{3}, t_{4} \rangle, 
\]
with differentials 
\[
dt_{0} = 0, \ \ dt_{1} = 0, \ \ dt_{2} = -a t_{1}, \ \ dt_{3} = - 2at_{2} - bt_{1}, \ \ dt_{4} = -3at_{3} - 2bt_{2} + ct_{1},
\]
and products 
\begin{align*}
t_{0} t_{1} &= 0, \ \ t_{0}t_{2} =at_{1}, \ \ t_{0}t_{3} = at_{2} + bt_{1}, \ \ t_{0}t_{4} = at_{3} + bt_{2} - ct_{1}, \ \ t_{1}t_{2} = at_{2}, \\ 
t_{1}t_{3} &= at_{3} + bt_{2}, \ \ t_{1}t_{4} = at_{4} + bt_{3} - ct_{2}, \ \ t_{2}t_{3} = bt_{3}, \ \ t_{2} t_{4} = bt_{4} - ct_{3}, \ \  t_{3}t_{4} = -ct_{4}.
\end{align*}
This algebra is quasi-isomorphic to $\Omega_{X}^{\bullet}(S_{4}(L))[-1]$ by Lemma \ref{inducedquasiiso}. Using the spectral sequence, we compute the cohomology algebra. It is given by 
\[
H^{\bullet}(C(\sigma)) = H^{\bullet}(C)t_{0} \bigoplus \big( \langle t_{1} \rangle \oplus \langle bt_{1}, bt_{2} + ct_{1}, at_{4} + \frac{1}{2}bt_{3} - ct_{2} \rangle \oplus \langle bct_{1}, act_{3} + bct_{2}, act_{4} + bct_{3} \rangle \oplus \langle abct_{4} \rangle\big).
\]
Computing the products of this algebra we find that only two of them are non-vanishing. They are given by 
\[
t_{1} \ast (at_{4} + \frac{1}{2}bt_{3} - ct_{2} ) = \frac{1}{6}bct_{1}, \qquad (at_{4} + \frac{1}{2}bt_{3} - ct_{2} )^2 = abct_{4}. 
\]
The dimensions of the cohomology of the associated $b^{5}$-algebroid are given by 
\[
\dim(H^{i}(A(\sigma))) = 1, \ \ 4, \ \ 7, \ \ 6, \ \ 2. 
\]
\subsection{Universal algebra for foliations} \label{UniversalAlgebrasection}
Let $\mathfrak{g}_{k} = Lie(G_{k,1})$, which has basis $z\partial_{z}, z^2 \partial_{z}, ..., z^{k} \partial_{z}$. Let $x_{0}, x_{1}, ..., x_{k-1}$ be the dual basis of $\mathfrak{g}_{k}^*$. Then the Chevalley-Eilenberg algebra of this Lie algebra is given by 
\[
CE(\mathfrak{g}_{k}) = \mathbb{R}[x_{0}, ..., x_{k-1}], 
\]
where each $x_{r}$ has degree $+1$ and the differential is given by 
\[
\delta(x_{r}) = \sum_{i = 0}^{r}(i-r)x_{i}x_{r-i}. 
\]
Now define 
\[
S_{k} = \oplus_{r = 0}^{k} CE(\mathfrak{g}_{k}) t_{r} =  \mathbb{R}[x_{0}, ..., x_{k-1}]\langle t_{0}, ..., t_{k} \rangle, 
\]
where each $t_{i}$ has degree $+1$. Define 
\[
\delta(t_{r}) = \sum_{i = 0}^{r}(i-r)x_{i}t_{r-i},
\]
and 
\[
t_{s} t_{r} = x_{s} t_{r} + x_{s+1}t_{r-1} + ... + x_{r-1}t_{s+1},
\]
for $s < r$. Then $S_{k}$ is a cdga defined over $CE(\mathfrak{g}_{k}) $.

Now let $A(\sigma)$ be a $b^{k+1}$-algebroid defined over $D \times \mathbb{R}$ with leaves $D$, $D \times \mathbb{R}^{+}$ and $D \times \mathbb{R}^{-}$. This is specified by a Maurer-Cartan element $\sigma \in \Omega_{D}^{1}(\mathfrak{g}_{k})$. Equivalently, this is defined by the data of a morphisms of cdgas 
\[
\psi_{\sigma} : CE(\mathfrak{g}_{k}) \to \Omega_{D}^{\bullet}, \qquad x_{i} \mapsto \eta_{i}. 
\]
This turns $\Omega_{D}^{\bullet}$ into an algebra over $CE(\mathfrak{g}_{k}) $, and this allows us to define a new cdga 
\[
S_{k} \otimes_{CE(\mathfrak{g}_{k})} \Omega_{D}^{\bullet}. 
\]
It is straightforward to check that this algebra is nothing but $\Omega^{\bullet}_{D}(S_{k}(L))[-1]$. Hence, $S_{k}$ is the universal algebra controlling the cohomology of $b^{k+1}$-algebroids. 
\subsection{Globalising the cohomology}\label{ssec:globalcohom}
Now let $D \subset M$ be a hypersurface in a manifold. The data $\sigma$ determines a $b^{k+1}$-algebroid $A(\sigma)$ on $M$. To make things concrete, we may choose a tubular neighbourhood $U$ of $D$ which is isomorphic to $L = \nu_{D}$, the normal bundle. By Lemma \ref{decompositionlemma}, we have an isomorphism of cdga's
\[
\Omega_{U}^{\bullet}(\sigma) \cong \Omega_{U}^{\bullet} \oplus \Omega_{D}^{\bullet - 1}(S_{k}(L))
\]
from which we obtain a surjective chain morphism $\Omega_{U}^{\bullet}(\sigma)  \to \Omega_{D}^{\bullet - 1}(S_{k}(L))$. By restricting forms we obtain a map
\[
\Res: \Omega_{M}^{\bullet}(\sigma) \to \Omega_{U}^{\bullet}(\sigma) \to \Omega_{D}^{\bullet - 1}(S_{k}(L)).
\]
which we think of as a generalized residue map. Note that this map depends on the choice of a tubular neighbourhood embedding. There is also an inclusion of smooth forms $\Omega_{M}^{\bullet} \to \Omega_{M}^{\bullet}(\sigma)$ which is also a chain map. 

\begin{remark}
The restricted algebroid $A(\sigma)|_{D}$ lives in the short exact sequence 
\[
0 \to L^{-k} \to A(\sigma)|_{D} \to TD \to 0,
\]
and this allows us to define a canonical higher residue map $\Res^{k} : \Omega_{M}^{\bullet}(\sigma) \to \Omega^{\bullet - 1}_{D}(L^{k})$ which extracts the component of an algebroid form that has a pole of order $k+1$ along $D$. The kernel of this residue consists of algebroid forms for the truncated algebroid $A(\sigma_{k-1})$. In other words, these are forms with poles of order at most $k$. Hence, by making use of the splittings $\tau_{r}$, we may iteratively extract the order $r$ residues of a form $\alpha \in  \Omega_{M}^{\bullet}(\sigma)$, for $0 \leq r \leq k$. These residues $\Res^{r}$ form the components of the map $\Res$ and they depend on both the algebroid $A(\sigma)$ as well as the choice of the tubular neighbourhood embedding. For a $b^{k+1}$-algebroid determined by a defining function for $D$, the $r^{th}$-residue $\Res^{r}$ just extracts the coefficient of $\frac{dz}{z^{r+1}}$.
\end{remark}

\begin{lemma}
There is a short exact sequence of chain complexes 
\[
0 \to \Omega_{M}^{\bullet} \to \Omega_{M}^{\bullet}(\sigma) \to \Omega_{D}^{\bullet - 1}(S_{k}(L)) \to 0,
\]
which induces a long exact sequence in cohomology 
\[
... \to H^{i}(M) \to H^{i}(A(\sigma)) \to H^{i-1}(D, S_{k}(L)) \to H^{i+1}(M) \to ... 
\]
\end{lemma}
\begin{proof}
The map $ \Omega_{M}^{\bullet} \to \Omega_{M}^{\bullet}(\sigma)$ is an inclusion and so is injective. Furthermore, since the map lands in smooth forms, these have zero projection to $\Omega_{D}^{\bullet - 1}(S_{k}(L))$. To see that $ \Omega_{M}^{\bullet}(\sigma) \to \Omega_{D}^{\bullet - 1}(S_{k}(L))$ is surjective, we choose a bump function $\psi_{U}$ which is supported in the tubular neighbourhood and is equal to $1$ in a neighbourhood of $D$. Now let $\alpha \in \Omega_{D}^{\bullet - 1}(S_{k}(L))$. By Lemma \ref{decompositionlemma}, this comes from an element $\tau(\alpha) \in \Omega_{U}^{\bullet}(\sigma)$, an algebroid form defined on $U$. The form $\psi_{U} \tau(\alpha)$ can be extended by zero to a global form in $ \Omega_{M}^{\bullet}(\sigma)$. And since $\psi_{U}$ is the identity on a neighbourhood of $D$, the singular part of $\psi_{U}\tau(\alpha)$ is $\alpha$. Hence $\Res(\psi_{U} \tau(\alpha)) = \alpha$, establishing surjectivity. Finally, to show exactness in the middle spot, note that a form $\omega$ with $\Res(\omega) = 0$ restricts to a form on $U$ with trivial singular part. But this is nothing but a smooth form. 
\end{proof}

\begin{proposition} \label{decomp1proof}
The connecting homomorphism $\delta : H^{i-1}(D, S_{k}(L)) \to H^{i+1}(M)$ vanishes. Therefore, we have isomorphisms 
\[
H^{i}(A(\sigma)) \cong H^{i}(M) \oplus H^{i-1}(D, S_{k}(L)),
\]
for all $i$. 
\end{proposition}
\begin{proof}
Let $\alpha \in \Omega_{D}^{i - 1}(S_{k}(L))$ be a closed form and let $\psi_{U} \tau(\alpha)$ be its extension to $M$. Then the connecting homomorphism is given by $\delta(\alpha) = d(\psi_{U} \tau(\alpha)). $ By Lemma \ref{technicalexactness} we can write $\tau(\alpha) = \alpha_{0} \wedge \gamma + d\eta$, with $\alpha_{0}$ and $\gamma$ smooth and closed forms pulled back from $D$ and $\eta$ a `horizontal' form with singularities along $D$. Hence 
\[
\delta(\alpha) = d \psi_{U} \wedge \alpha_{0} \wedge \gamma + d \psi_{U} \wedge d \eta = d(\psi_{U} \alpha_{0} \wedge \gamma - d\psi_{U} \wedge \eta).  
\] 
Note that $\psi_{U} \alpha_{0} \wedge \gamma$ extends by zero to a global form on $M$. Second, $\eta$ is a form with singularities along $D$, but since $\psi_{U}$ is identically equal to $1$ in a neighbourhood of $D$, its derivative $d\psi_{U}$ vanishes in this neighbourhood, and hence $d\psi_{U} \wedge \eta$ is smooth around $D$ and can be extended by zero to a global smooth form on $M$.  
Let 
\[
\chi(\alpha, \psi_{U}, h) = \psi_{U} \alpha_{0} \wedge \gamma - d\psi_{U} \wedge \eta,
\] 
a smooth global form which depends on the choices of $\alpha$, a metric $h$ on $L$, a tubular neighbourhood, and the bump function $\psi_{U}$. Then $\delta(\alpha) = d \chi$, so that the connecting homomorphism vanishes. 
\end{proof}

Let's take a closer look at the residue map $\Res$. It factors as $\Res = \pi_{2} \circ r_{1}$, where 
\[
r_{1} : \Omega_{M}^{\bullet}(\sigma) \to \Omega_{U}^{\bullet}(\sigma) \cong \Omega_{U}^{\bullet} \oplus \Omega_{D}^{\bullet}(S_{k}(L))[-1],
\]
and the second map $\pi_{2}$ is the projection onto the second factor. The first map is a morphism of cdga's and induces on cohomology the algebra map 
\[
H(r_{1}) : H^{\bullet}(A(\sigma)) \to H^{\bullet}(D) \oplus H^{\bullet}(D, S_{k}(L))[-1].
\]
Now $H^{\bullet}(M) \subset H^{\bullet}(A(\sigma))$ is a subalgebra and if we restrict the map $H(r_{1})$ to $\beta \in H^{\bullet}(M)$ we get 
\[
H(r_{1})(\beta) = (\beta|_{D}, 0) \in H^{\bullet}(D) \oplus H^{\bullet}(D, S_{k}(L))[-1]. 
\]
Note that the restriction map $H^{\bullet}(M) \to H^{\bullet}(D)$ is a homomorphism and in this way any $H^{\bullet}(D)$-module becomes a $H^{\bullet}(M)$-module. In particular, $H^{\bullet}(D, S_{k}(L))[-1]$ is in this way an $H^{\bullet}(M)$-module. 
\begin{proposition}
The sequence 
\[
0 \to H^{\bullet}(M) \to H^{\bullet}(A(\sigma)) \to H^{\bullet}(D, S_{k}(L))[-1] \to 0
\]
is a short exact sequence of $H^{\bullet}(M)$-modules. \end{proposition}
\begin{proof}
The sequence is exact by Proposition \ref{decomp1proof}. Note first that $H^{\bullet}(M) \to H^{\bullet}(A(\sigma))$ is a morphism of algebras and hence in particular, a morphism of $H^{\bullet}(M)$-modules. To see that $H(\Res)$ is a module morphism, we compute, for $\beta \in H^{\bullet}(M) $ and $\omega \in H^{\bullet}(A(\sigma))$ : 
\begin{align*}
H(\Res)(\beta \wedge \omega) &= \pi_{2}(H(r_{1})(\beta \wedge \omega)) \\ 
&= \pi_{2}( H(r_{1})(\beta) \wedge H(r_{1})(\omega)) \\
&= \pi_{2}( (\beta|_{D}, 0) \wedge (\tilde{\omega}, H(\Res)(\omega)) ) \\
&= \pi_{2}( \beta|_{D} \wedge \tilde{\omega}, \beta|_{D} \wedge H(\Res)(\omega)) \\
&= \beta|_{D} \wedge H(\Res)(\omega) \\
&= \beta \ast H(\Res)(\omega). 
\end{align*}
The first equality uses the factorization, the second uses the fact that $H(r_{1})$ is an algebra homomorphism, the third uses the above calculation of the map, the fourth uses Proposition \ref{localalgebraformula}, and the last equality uses the definition of the action of $H^{\bullet}(M)$ on $H^{\bullet}(D, S_{k}(L))[-1]$. 
\end{proof}

\begin{remark}
One may expect that the decomposition $H^{\bullet}(A(\sigma)) = H^{\bullet}(M) \oplus H^{\bullet-1}(D, S_{k}(L))$ can be made compatibly with the algebra structures, generalizing Proposition \ref{localalgebraformula}. At present we do not know whether this is true.  
\end{remark}

\subsection{Restriction to the boundary}
Recall that the restricted algebroid $A(\sigma)|_{D}$ sits in the following short exact sequence 
\[
0 \to L^{-k} \to A(\sigma)|_{D} \to TD \to 0.
\]
Dualizing and taking wedge powers gives the following short exact sequence of complexes
\begin{equation}\label{eq:sesboundary}
0 \to \Omega_{D}^{\bullet} \to \Omega^{\bullet}_{A(\sigma)|_{D}} \to \Omega^{\bullet - 1}_{D} \otimes L^{k} \to 0. 
\end{equation}
It can be useful to express the restriction morphism $\mathcal{R}: \Omega_{M}^{\bullet}(\sigma) \to \Omega^{\bullet}_{A(\sigma)|_{D}} $ in terms of the decomposition of Lemma \ref{decompositionlemma}. First, the subcomplex $\Omega_{M}^{\bullet}$ maps to the subcomplex $\Omega_{D}^{\bullet} \subset \Omega^{\bullet}_{A(\sigma)|_{D}}$ via the restriction map.

Next, we focus on the subcomplex $\Omega_{D}^{\bullet - 1}(S_{k}(L))$. Note that $S_{k-1}(L) \subset S_{k}(L)$ is a subbundle which is preserved by the connection. Hence, we get a short exact sequence of complexes 
\[
0 \to \Omega_{D}^{\bullet - 1}(S_{k-1}(L)) \to \Omega_{D}^{\bullet - 1}(S_{k}(L)) \to \Omega_{D}^{\bullet - 1} \otimes L^{k} \to 0. 
\]
\begin{lemma}\label{lem:restrictmap}
The restriction map restricts to the following map 
\[
\mathcal{R} : \Omega_{D}^{\bullet - 1}(S_{k-1}(L)) \to \Omega_{D}^{\bullet}, \qquad \sum_{r = 0}^{k-1} \alpha_{r} \mapsto \sum_{r = 0}^{k-1} \alpha_{r} \wedge \eta_{r}.
\]
Therefore, we obtain a map of short exact sequences 
\[
\begin{tikzpicture}[scale=1.5]
\node (A) at (-0.5,1) {$0$};
\node (B) at (1,1) {$  \Omega_{D}^{\bullet - 1}(S_{k-1}(L))$};
\node (C) at (2.75,1) {$ \Omega_{D}^{\bullet - 1}(S_{k}(L)) $};
\node (D) at (4.5,1) {$ \Omega_{D}^{\bullet-1}\otimes L^k$};
\node (E) at (6,1) {$0$};
\node (A1) at (-0.5,0) {$0$};
\node (B1) at (1,0) {$\Omega_{D}^{\bullet}$};
\node (C1) at (2.75,0) {$\Omega^{\bullet}_{A(\sigma)|_{D}}$};
\node (D1) at (4.5,0) {$\Omega_{D}^{\bullet-1}\otimes L^k$};
\node (E1) at (6,0) {$0$};

\path[->,>=angle 90]
(A) edge (B)
(B) edge (C)
(C) edge (D)
(D) edge (E)
(A1) edge (B1)
(B1) edge (C1)
(C1) edge (D1)
(D1) edge (E1)
(B) edge node[left]{$\mathcal{R}$} (B1)
(C) edge node[left]{$\mathcal{R}$} (C1)
(D) edge node[right]{$=$} (D1);

\end{tikzpicture}
\]
\end{lemma}
As a result, we obtain a Gysin-like long exact sequence computing the cohomology of $A(\sigma)|_{D}$. 
\begin{proposition}
There is a long exact sequence of cohomology groups 
\[
... \to H^{i}_{D} \to H^{i}(A(\sigma)|_{D}) \to H^{i-1}_{D}(L^{k}) \to H^{i+1}_{D} \to ...
\]
with connecting homomorphism given by the wedge product with the extension class $c(\sigma) \in H^{2}_{D}(L^{-k})$: 
\[
\delta = (-1)^{i}c(\sigma) \wedge \ : H^{i-1}_{D}(L^{k}) \to H^{i+1}_{D}.
\]
\end{proposition}
\begin{proof}
By Lemma \ref{lem:restrictmap} we may compute the connecting homomorphism as $\mathcal{R} \circ \delta$, where $\delta$ is the connecting homomorphism of the short exact sequence for $\Omega_{D}^{\bullet}(S_{k}(L))$. This later sequence is split as a graded vector space. Hence, given $\omega \in H^{i-1}_{D}(L^{k})$, the connecting homomorphism is computed as 
\begin{align*}
\mathcal{R} \circ \delta(\omega) &= \mathcal{R}( d \omega) = \mathcal{R}(\sum_{j = 1}^{k-1} (j - r) \eta_{j} \wedge \omega) \\
&= (-1)^{i-1} (\sum_{j = 1}^{k-1} (j - r) \eta_{j} \wedge \eta_{k-j}) \wedge \omega \\
&= (-1)^{i} c(\sigma) \wedge \omega.
\end{align*}
In the second equality we used the expression for the differential from Equation \ref{Anticanonicaldifferential}, using the fact that $d^{\nabla}\omega = 0$. In the last equality we used the expression for $c(\sigma)$ from Theorem \ref{extensionclasssection2}. 
\end{proof}

Given a symplectic form $\omega \in \Omega^2(A(\sigma))$, we can use the above result to decompose the cohomology class $[\omega|_D]$. This allows us to study the geometry induced on the hypersurface $D$, as in Proposition \ref{trivialvariationfoliationinduced} and Proposition \ref{lem:structureonD}.

\appendix

\section{Group and Lie algebroid extensions}\label{sec:appa}
In this section we recall some general theory for the extension of groups, Lie algebroids and Lie groupoids. For a detailed account on Lie algebroid extensions see \cite{BRAHIC2010352}, or \cite{cra03}.
\subsection{Group extensions}
We consider group extensions of the form
\begin{equation*}
1 \rightarrow A \rightarrow B \rightarrow C \rightarrow 1,
\end{equation*}
with $A$ Abelian. Let $\sigma : C \rightarrow B$ be any set-theoretic splitting, then we may define a $C$-module structure:
\begin{equation*}
c\cdot a := \sigma(c)\cdot a\cdot \sigma(c)^{-1}
\end{equation*}
This structure does not depend on the choice of splitting as $A$ is abelian. Consequently, one may consider the complex of group-cochains with values in $A$: $\Omega^k(C;A)$.

Given any splitting $\sigma : C\rightarrow B$ we may consider
\begin{equation*}
C_{\sigma}(c_1,c_2) = \sigma(c_1c_2)^{-1}\sigma(c_1)\sigma(c_2).
\end{equation*}
It then follows that using $\sigma$, $B$ is isomorphic to $C \ltimes_{C_{\sigma}} A$ which is $C \times A$ as a set and has group structure given by
\begin{equation*}
(c_1,a_1)\cdot (c_2,a_2) = (c_1 \cdot c_2, c_1^{-1} \cdot a_1 \cdot a_2 \cdot (c_1c_2)^{-1} \cdot C_{\sigma}(c_1,c_2)).
\end{equation*}
The cohomology class $[C_{\sigma}] \in H^2(C;A)$ is independent on the choice of splitting, and vanishes precisely when the extension is split, that is $B$ is isomorphic to $A \rtimes C$ as groups.

\subsection{Lie algebroid extensions}

The following result is well-known, but for completeness-sake we will still provide a proof:
\begin{proposition}\label{prop:algsplittingsabstract}
Suppose that we have a short exact sequence of Lie algebroids:
\begin{equation*}
0 \rightarrow B \rightarrow A \rightarrow C \rightarrow 0,
\end{equation*}
with $B$ contained in $\ker \rho(A)$. Then:
\begin{enumerate}
\item Any vector bundle splitting $\sigma : C \rightarrow A$ induces a $C$-connection on $B$, given by $\nabla_c(b) := [\sigma(c),b]$.
\item If $\sigma$ is a bracket preserving splitting $\Omega^{\bullet}(C;B)$ obtains the structure of a dgla. If $\tilde{\sigma}$ is any other bracket preserving splitting then $\sigma - \tilde{\sigma} \in {\rm MC}(\Omega^{\bullet}(C,B))$.
\item When $B$ is abelian, the representation $\nabla$ does not depend on the choice of splitting $\sigma$.
\end{enumerate}
\end{proposition}
\begin{proof}
1) We have that
\begin{equation*}
\nabla_c(fb) = [\sigma(c),fb] = f[\sigma(c),b] + \mathcal{L}_{\rho_A(\sigma(c))}(b) = f[\sigma(c),b] + \mathcal{L}_{\rho_C(c)}(b),
\end{equation*}
and
\begin{equation*}
\nabla_{fc}(b) = [f\sigma(c),b] = f[\sigma(c),b] - \mathcal{L}_{\rho_B(b)}(f)c = f[\sigma(c),b],
\end{equation*}
as $B$ has trivial anchor.\\
2) We have to show that
\begin{equation*}
d^{\nabla}(\sigma-\tilde{\sigma}) + \frac{1}{2}[\sigma-\tilde{\sigma},\sigma-\tilde{\sigma}] = 0.
\end{equation*}
We have
\begin{align*}
d^{\nabla}(\sigma-\tilde{\sigma})(X,Y) &= \nabla_X(\sigma(Y)-\tsigma(Y))-\nabla_Y(\sigma(X)-\tsigma(Y)) - \sigma([X,Y]) + \tsigma([X,Y])\\
&=[\sigma(X),\sigma(Y)-\tsigma(Y)]-[\sigma(Y),\sigma(X)-\tsigma(X)] - \sigma([X,Y]) + \tsigma([X,Y])\\
&= [\sigma(X),\sigma(Y)]-[\sigma(X),\tsigma(Y)]-[\sigma(Y),\sigma(X)]+[\sigma(Y),\tsigma(X)]-\sigma([X,Y])+\tsigma([X,Y])\\
&= C_{\sigma}(X,Y) + [\sigma(X),\sigma(Y)] - [\sigma(X),\tsigma(Y)]+[\sigma(Y),\tsigma(X)] + \tsigma([X,Y]),
\end{align*}
whereas
\begin{align*}
\frac{1}{2}[\sigma-\tilde{\sigma},\sigma-\tilde{\sigma}] &= [\sigma(X)-\tsigma(X),\sigma(Y)-\tsigma(Y)]\\
&= [\sigma(X),\sigma(Y)]-[\sigma(X),\tsigma(Y)]-[\tsigma(X),\sigma(Y)]+[\tsigma(X),\tsigma(Y)].
\end{align*}
From this we conclude that 
\begin{equation*}
d^{\nabla}(\sigma-\tilde{\sigma}) - \frac{1}{2}[\sigma-\tilde{\sigma},\sigma-\tilde{\sigma}] = C_{\sigma} + C_{\tsigma}.
\end{equation*}
3) 
As $\sigma(c)-\tsigma(c)$ lies in $B$ as the sequence is exact it follows that $[\sigma(c)-\tsigma(c),b] = 0$ for all $b \in B$. Consequently $\nabla_c(b)$ does not depend on the choice of splitting.
\end{proof}

Given a short exact sequence of Lie algebroids, and a vector bundle splitting $\sigma : C \rightarrow A$, we may define the curvature of $\sigma$ by:
\begin{equation*}
C_{\sigma} \in \Omega^2(C;B), \quad C_{\sigma}(X,Y) := [\sigma(X),\sigma(Y)] - \sigma([X,Y]).
\end{equation*} 

\begin{proposition}
Given a vector bundle splitting $\sigma : C \rightarrow A$ of a short exact sequence of Lie algebroids, the curvature $C_{\sigma}$ defines a cohomology class in $H^2(C;B)$, where the representation is induced by the above proposition.
\begin{itemize}
\item $C: = [C_{\sigma}] \in H^2(C;B)$ is independent from the choice of splitting.
\end{itemize}

\end{proposition}

\section{Characteristic classes of (symplectic) foliations}
We recall some theory of characteristic classes of (symplectic) foliations.

Let $\Omega^{\bullet}_{\ff}(M)$ denote the kernel of the restriction map $\iota^* : \Omega^{\bullet}(M) \rightarrow \Omega^{\bullet}(\ff)$. We thus obtain a short exact sequence of cochain complexes:
\begin{equation}\label{eq:folseq}
0 \rightarrow \Omega_{\ff}^{\bullet}(M) \rightarrow \Omega^{\bullet}(M) \rightarrow \Omega^{\bullet}(\ff) \rightarrow 0.
\end{equation}
We also consider $\Omega^{\bullet}(\ff,\nu^*)$, endowed with the differential induced by the Bott-connection. We have several maps relating these complexes, the first is:
\begin{align}
p : \Omega^{k+1}_{\ff}(M) &\rightarrow \Omega^k(\ff,\nu^*), \label{eq:p}\\
p(\alpha)(X_1,\ldots,X_k)(\bar{N}) &= \alpha(X_1,\ldots,X_n,N), \nonumber
\end{align}
with $X_1,\ldots,X_n \in \Gamma(T\ff)$ and where for $N \in \mathfrak{X}(M)$, $\bar{N}$ denotes the corresponding section of the zero-section. Another map, which we denote simply by the differential, is given by
\begin{align*}
d : H^k(\ff) &\rightarrow H^{k+1}_{\ff}(M),\\
[\alpha] &\mapsto [d\tilde{\alpha}],
\end{align*}
where $\tilde{\alpha} \in \Omega^k(M)$ is any extension of $\alpha$. We may combine both these maps to obtain the variation map $d_{\nu} = p \circ d : H^k(\ff) \rightarrow H^{k}(\ff,\nu^*)$.

\begin{definition}
The \textbf{symplectic variation} of a symplectic foliation $(\ff,\omega_{\ff})$ is the class ${\rm var}_{\omega_{\ff}} := d_{\nu}[\omega_{\ff}] \in H^2(\ff,\nu^*)$.
\end{definition}

Let $\theta \in \Omega^1(M;\nu)$ denote the quotient map arising from the sequence
\begin{equation*}
0 \rightarrow T\ff \rightarrow TM \rightarrow \nu \rightarrow 0.
\end{equation*}
One may phrase the involutivity condition for $\ff$ purely in terms of $\theta$ as follows:
\begin{lemma}
Let $\theta \in \Omega^1(M;\nu)$ and a flat connection $\nabla$ on $\nu$ be given. Then $\ker\theta$ is involutive if and only if
\begin{equation*}
d^{\nabla}\theta = \beta \wedge \theta,
\end{equation*}
for some $\beta \in \Omega^1(M)$.
\end{lemma}

\begin{lemma}
Let $\ff$ be a foliation on $M$, and $\nabla$ a flat connection on the normal bundle $\nu$ up to gauge equivalence. Then 
\begin{equation*}
m(\ff,\nabla) := [\beta|_{\ff}] \in H^1(\ff), 
\end{equation*}
is a well-defined cohomology class, called the \textbf{modular class} of $(\ff,\nabla)$.
\end{lemma}
Let 
\begin{equation*}
\delta : H^1(\ff) \rightarrow H^2_{\ff}(D) \rightarrow H^1(\ff,\nu^*)
\end{equation*}
be the connection morphism of \eqref{eq:folseq} combined with the map $p$ of Equation \ref{eq:p}.
\begin{definition}\label{def:varfol}
The \textbf{variation} of $\ff$ is defined as $\Var_{\ff} := \delta(m(\ff,\nabla))$.
\end{definition}
\begin{lemma}
The variation $\Var_{\ff}$ does not depend on the choice of connection.
\end{lemma}
\begin{proof}
Changing to another flat connection means 
\[
\tilde{\nabla} = \nabla + \gamma, 
\]
for $\gamma \in \Omega^{1, cl}_{X}$. Then 
\[
d^{\tilde{\nabla}}\theta = (\beta + \gamma) \wedge \theta. 
\]
Hence, $m(F, \tilde{\nabla}) = (\beta + \gamma)|_{F} = m(F, \nabla) + \gamma|_{F}$. Therefore, 
\[
\delta(m(F, \tilde{\nabla})) = \delta(m(F, \nabla)) + d\gamma = \delta(m(F, \nabla)). 
\]
\end{proof}

\begin{proposition}\label{prop:bottextend}
Consider a pair $(F, \nabla)$, let $\theta \in \Omega^1_{X} \otimes \nu$, and let $m \in H^1_{F}$ be the modular class. Then the following are equivalent
\begin{enumerate}
\item $\mathrm{var}(F) = \delta(m) = 0$, 
\item there exists a flat connection $\tilde{\nabla}$ on $\nu$ such that $d^{\tilde{\nabla}}\theta = 0$, 
\item there exists a flat extension of the Bott connection on $\nu$. 
\end{enumerate}
Furthermore, if these conditions are equivalent, the flat extension of the Bott connection is unique up to an element of $H^0_{F}(\nu^*)$. In other words, it is unique up to a flat section of the Bott connection on $\nu^*$. 
\end{proposition}
\begin{proof}
First from $1$ to $2$. Assume that $\delta(m) = 0$. This means that the class $[d\beta] \in H^{1}_{F}(\nu^*)$ is $0$, implying that $d\beta$ is in fact exact in $\Omega^{\bullet}_{F} \otimes \nu^*$. This means that there is some $\eta \in \nu^*$ with $d\eta = d\beta$. Therefore, $\tilde{\beta} = \beta - \eta \in \Omega_{X}^1$ is closed. But \[
\tilde{\beta} \wedge \theta = \beta \wedge \theta - \eta \wedge \theta = \beta \wedge \theta. 
\]
Now since $\tilde{\beta}$ is closed, we may define a new flat connection $\tilde{\nabla} = \nabla - \tilde{\beta}$ on $\nu$. With respect to this connection 
\[
d^{\tilde{\nabla}}\theta = d^{\nabla} \theta - \tilde{\beta} \wedge \theta = \beta \wedge \theta - \beta \wedge \theta = 0. 
\]
Now we show $2$ implies $1$. So let $\tilde{\nabla}$ be a flat connection such that $d^{\tilde{\nabla}}\theta = 0$. It will differ from $\nabla$ by a closed $1$-form: $\tilde{\nabla} = \nabla - \gamma$, for $\gamma \in \Omega^{1,cl}_{X}$. Hence 
\[
0 = d^{\tilde{\nabla}}\theta  = d^{\nabla} \theta - \gamma \wedge \theta = \beta \wedge \theta - \gamma \wedge \theta. 
\]
Therefore, $\beta \wedge \theta = \gamma \wedge \theta$. This means that $\beta|_{F} = \gamma|_{F}$ and so 
\[
\delta(m) = [ d \gamma ] = 0. 
\]

To show that $2$ and $3$ are equivalent we prove that a connection $\nabla$ on $\nu$ extends the Bott connection if and only if $d^{\nabla} \theta = 0$. Recall that the Bott connection is a $TF$ connection $\nabla^B$ on $\nu$ defined as follows 
\[
\nabla^{B}_{V}(W) = \theta([V, \tilde{W}]),
\]
where $V \in TF$ and $W \in \nu$ which has an extension $\tilde{W} \in TX$. Therefore, a flat $TX$-connection $\nabla$ on $\nu$ extends the Bott connection if and only if  
\[
\nabla_{V}(\theta(U)) = \theta([V, U])
\]
for $V \in F$ and $U \in TX$. But since $\theta(V) = 0$, this equation may be rewritten as follows: 
\[
d^{\nabla}\theta(V, U) = \nabla_{V}(\theta(U))  - \nabla_{U}(\theta(V)) - \theta([V, U]) = 0. 
\]
Hence, $\nabla$ extends Bott if and only if $d^{\nabla}\theta$ vanishes when one input lies in $F$. But since $\nu$ has rank $1$, this is precisely the condition that $d^{\nabla}\theta = 0$
\end{proof}

\begin{definition}
A corank $1$ foliation $F \subset TX$ on $X$ has trivial variation if $\mathrm{var}(F) = 0$. 
\end{definition}

Now let $F$ be a foliation with trivial variation. Let $\theta \in \Omega_{X}^1(\nu)$ be the defining $1$-form and let $\nabla$ be a flat connection such that $d^{\nabla}\theta = 0$. Now let $\alpha \in H^{0}_{F}(\nu^*)$. Namely, $\alpha \in \nu^* \subset \Omega_{X}^1$ such that $d\alpha = 0$. Then $\alpha \wedge \theta = 0$ and hence $d^{\tilde{\nabla}}(\theta) = 0$, where $\tilde{\nabla} = \nabla + \alpha$. 

Let $t_{\alpha} \in \nu^*$ be the section $\alpha$ but viewed as a section of the line bundle rather than a $1$-form. Then 
\[
\alpha = \langle t_{\alpha}, \theta \rangle,
\]
and 
\[
0 = d\alpha = \langle d^{\nabla} t_{\alpha}, \theta \rangle + \langle t_{\alpha}, d^{\nabla} \theta \rangle = \langle d^{\nabla} t_{\alpha}, \theta \rangle . 
\]
This means that $d^{\nabla}t_{\alpha} \in \nu^* \otimes \nu^*$. Hence, 
\[
d^{\nabla}t_{\alpha} =  \langle u_{\alpha}, \theta \rangle,
\]
with $u_{\alpha} \in \nu^{-2}$.

Now suppose that $[\theta] = 0 \in H^1_{X}(\nu, \nabla)$. This means that there is some section $s \in \nu$ such that 
\[
\nabla(s) = \theta. 
\]

\bibliographystyle{alpha}
\bibliography{references} 
\end{document}